\documentclass{amsart}
\input{style.sty}

\newtheorem{lemma}[theorem]{Lemma}

\theoremstyle{definition}
\newtheorem{definition}[theorem]{Definition}

\theoremstyle{remark}
\newtheorem{remark}[theorem]{Remark}

\numberwithin{equation}{section}

\begin{document}

\title{On Heegaard Floer minimal knots in sutured manifolds}

\author{Fraser Binns}
\address{Department of Mathematics, Princeton University}
\email{fb1673@princeton.edu}


\thanks{FB was supported by the Simons Grant {\em New structures in low-dimensional topology}}

%    General info

\date{\today}


\begin{abstract}
  Li-Xie-Zhang classified instanton Floer minimal knots in balanced sutured manifolds subject to a condition on the fundamental group. In this paper, we give a similar classification in the Heegaard Floer homology setting. Since our classifications agree when they are both applicable, this provides further evidence for the conjecture of Kronheimer-Mrowka that instanton Floer homology and Heegaard Floer homology are isomorphic. We also study link Floer homology botany question in $S^1\times S^2$, showing that link Floer homology detects spherical braid closures among homologically nontrivial links.
\end{abstract}

\maketitle

Heegaard Floer homology is a powerful package of invariants due to Ozsv\'ath-Szab\'o~\cite{ozsvath2004holomorphic}. To each $3$-manifold it assigns a vector space $\widehat{\HF}(Y)$. To each knot $K$ in a $3$-manifold $Y$ it assigns another vector space called \emph{knot Floer homology},  $\widehat{\HFK}(K,Y)$, due independently to J. Rasmussen~\cite{Rasmussen} and Ozsv\'ath-Szab\'o~\cite{Holomorphicdisksandknotinvariants}. The knot Floer homology of a rationally null-homologous knot, $K$, in a $3$-manifold, $Y$, is related to the Heegaard Floer homology of $Y$ by the rank bound:
\begin{equation}\label{eq:rankbound}
    \rank(\widehat{\HFK}(K,Y))\geq \rank(\widehat{\HF}(Y)).
\end{equation}
See Section~\ref{sec:review} for some discussion. In this paper we are broadly interested in the question of classifying links with fixed knot Floer homology; the \emph{botany question}. In attempting to address the botany question in general $3$-manifolds, it is natural to start with those knots for which rank inequality~(\ref{eq:rankbound}) is tight, namely \emph{(Heegaard) Floer simple knots}. These knots also arise in the context of the Berge conjecture; see~\cite{hedden2011floer}. The only (Heegaard) Floer simple knot in $S^3$ is the unknot, which follows from work of Ozsv\'ath-Szab\'o~\cite[Theorem 1.2]{ozsvath2004genusbounds}. Examples of other Floer simple knots can be obtained as the cores of $n>2g(K)-1$-surgeries on a type of knot called an ``$L$-space knots $K$"; a result of Hedden~\cite[Theorem 1.4]{hedden2011floer}. The class of Heegaard Floer simple knots in $3$-manifolds is also closed under an appropriate version of the connect sum operation; see~\cite{ni2014homological}.

Juh\'asz extended Heegaard Floer homology to a class of $3$-manifolds with boundary called \emph{balanced sutured manifolds}~\cite{juhasz2006holomorphic}. A \emph{sutured manifold} $(Y,\gamma)$ is an oriented manifold $Y$ equipped with a decomposition of $\partial Y=R_+(\gamma)\cup\gamma\cup R_-(\gamma)$ that satisfies various conditions; see Section~\ref{sec:review} for further details. Such manifolds were introduced by Gabai in the study of taut  foliations~\cite{gabai1983foliations}. To each sutured manifold, $(Y,\gamma)$, sutured Floer homology associates a vector space $\SFH(Y,\gamma)$. If $K$ is a knot in $(Y,\gamma)$, then the exterior of $K$ in $Y$ is naturally endowed with the structure of a sutured manifold $(Y(K),\gamma(K))$. Note that $H_1(\partial Y;\Q)$ maps naturally to $H_1(Y;\Q)$. If $[K]=0\in H_1(Y;\Q)/H_1(\partial Y;\Q)$ then we have the following generalization of inequality~(\ref{eq:rankbound}):
\begin{equation}\label{eq:rankbound2}
    \rank(\SFH(Y(K),\gamma(K)))\geq 2\rank(\SFH(Y,\gamma)).
\end{equation}

See Lemma~\ref{lem:spectral} for further details. Our main result reduces the problem of determining when inequality~(\ref{eq:rankbound2}) is tight to the problem of determining when inequality~(\ref{eq:rankbound}) is tight.

\begin{restatable}{theorem}{main}\label{thm:main}
    Suppose that $K$ is a knot in a balanced sutured manifold $(Y,\gamma)$ with $[K]=0\in H_1(Y;\Q)/H_1(\partial Y;\Q)$. If $\rank(\SFH(Y(K),\gamma(K)))=2\cdot\rank(\SFH(Y,\gamma))\neq 0$ then $Y$ splits as a connect sum $M\# (Y',\gamma)$ with $M$ a closed $3$-manifold and $K$ a Floer simple knot in $M$.
\end{restatable}

\noindent We defer the definition of balanced to Section~\ref{sec:review}.

There is another package of invariants in low dimensional topology that predates Heegaard Floer homology called \emph{instanton Floer homology}; see~\cite{FLoerinstanton3manifolds,Floerinstatonhomologysurgeryknots} for two archetypal instances. This package has seen substantial development by Kronheimer-Mrowka~\cite{kronheimer2010knots,kronheimer2011khovanov,kronheimer2011knot} as well as a number of other authors. Kronheimer-Mrowka conjecture that appropriate versions of instanton Floer homology and Heegaard Floer homology are isomorphic~\cite[Conjecture 7.24]{kronheimer2010knots}. Li-Xie-Zhang show that if $K$ is a knot in the connect sum $M\#(Y,\gamma)$ of a closed manifold $M$ and a sutured manifold $(Y,\gamma)$ satisfying an instanton Floer homology analogue of inequality~(\ref{eq:rankbound2}) as well as some other conditions, then $K$ is~\emph{instanton} Floer simple in $M$~\cite[Theorem 1.4]{li2022floer}. Here, a knot is~\emph{instanton Floer simple} if the instanton Floer homology version of inequality~(\ref{eq:rankbound}) is tight. Li-Xie-Zhang's proof is dependent on various techniques available in instanton Floer homology that do not currently have analogues in Heegaard Floer homology. In particular, they use a version of sutured instanton Floer homology with local coefficients --- see~\cite[Section 2.5]{li2022floer} --- as well as a generalized version of Kronheimer-Mrowka's $\Inat(Y,K)$~\cite{kronheimer2011khovanov} --- see~\cite[Section 3.2]{li2022floer}. Theorem~\ref{thm:main} can be viewed as a version of Li-Xie-Zhang's result in the Heegaard Floer setting. In particular, the combination of the two results provides further evidence for the equivalence of Heegaard Floer homology and instanton Floer homology, at least when the hypotheses of both Theorem~\ref{thm:main} and~\cite[Theorem 1.4]{li2022floer} apply.

 Several classification results follow from Theorem~\ref{thm:main}. For example:

\begin{corollary}
    Suppose that $L$ is an $n$-component link in $S^3$ containing a component $T(2,2m+1)$ for some $m$. If $\rank(\widehat{\HFK}(L))\leq 2^{n-1}|2m+1|$ then $L$ is the split sum of an $n-1$ component unlink and $T(2,2m+1)$.
\end{corollary}
\begin{proof}
    
Suppose that $L$ is as in the statement of the Corollary. Fix a component $K$ of $L$ that is $T(2,2m+1)$. Inequality~(\ref{eq:rankbound2}) gives rank bounds $\rank(\widehat{\HFK}(L_i))\geq 2\rank(\widehat{\HFK}(L_{i-1}))$ for $i\in\{n,n-1,\dots 2\}$, where $L_{i-1}$ is obtained from $L_{i}$ by removing a component other than $K$ and $L_n:=L$.  Since $\rank(\widehat{\HFK}(T(2,2m+1)))=|2m+1|$, each inequality $\rank(\widehat{\HFK}(L_i))\geq 2\rank(\widehat{\HFK}(L_{i+1}))$ is tight. An application of Theorem~\ref{thm:main} shows that $L_{i+1}$ is obtained from $L_i$ by adding a split unknotted component. This concludes the proof.\end{proof}

 To place these results in context, recall that Ni proved that the $n$-component unlink is the only link $L$ with ${\rank(\widehat{\HFK}(L))=2^{n-1}}$\cite[Proposition 1.4]{ni2014homological}, while Kim showed that the only $n$ component link $L$ with $\rank(\widehat{\HFK}(L))=2^{n}$ is the split sum of an $n-2$ component unlink and the Hopf link~\cite[Theorem 1]{kim2020links}.

Finally, we give some new classification results for links in $S^1\times S^2$. To state them, recall that a spherical braid closure is a link in $S^1\times S^2$ can be obtained from a spherical braid, $\beta$, by taking its closures, as indicated in Figure~\ref{fig:sphericalbraidclosure}. We denote the closure of the spherical braid $\beta$ by $\widehat{\beta}$.

 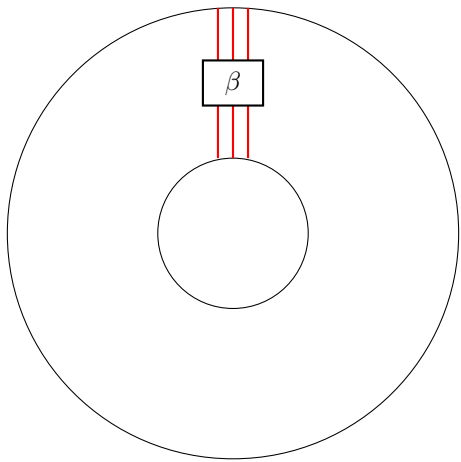
\begin{figure}[h]
    \centering

    \begin{tikzpicture}
        % Outer sphere
         (0,0) circle (3);
        \draw (0,0) circle (3);

        % Inner sphere (cut-out)
         (0,0) circle (1);
        \draw (0,0) circle (1);

        % Arcs connecting inner and outer boundary
        \draw[thick,red] (-0.2,1) -- (-0.2,3);
        \draw[thick,red] (0.2,1) -- (0.2,3);
         \draw[thick,red] (0,1) -- (0,3);

        \draw[fill=white, thick] (-0.4,1.7) rectangle (0.4,2.3);
        \node at (0,2) {\(\beta\)};
    \end{tikzpicture}

    \caption{The closure of a spherical braid is obtained from a spherical braid $\beta\subset S^2\times[-1,1]$ by identifying $S^2\times\{\pm 1\}$. In the figure, the outer and inner circles indicate $S^2\times\{\pm 1\}$.}
    \label{fig:sphericalbraidclosure}
\end{figure}

\begin{restatable}{corollary}{sphericalthreebraids}\label{cor:sphericalthreebraids}
     Suppose that $L$ is a homologically non-trivial, non-split link in $S^1\times S^2$ with ${\widehat{\HFL}(L)\cong\widehat{\HFL}(\widehat{\beta})}$ where $\widehat{\beta}$ is any of the nine spherical $n$-braid closures with $n\leq 3$. Then $L$ is a spherical braid closure of a braid $\alpha$ of the same index as $\beta$.
\end{restatable}

The hypothesis that $L$ has irreducible exterior is readily removed by an application of the K\"unneth formula, Equation~(\ref{eq:Kunneth}). Corollary~\ref{cor:sphericalthreebraids} follows from the classification of spherical $3$-braids up to conjugation, which we recall in Lemma~\ref{lem:3braidss1s2}, as well as the following result:

\begin{restatable}{theorem}{sphericalbraiddetection}\label{thm:sphericalbraiddetection}
     Suppose that $L$ is an $n$-component homologically non-trivial link in $S^1\times S^2$ with irreducible exterior and $|L\cap (\{*\}\times S^2)|\neq 1$. Then $\widehat{\HFK}(L,S^1\times S^2)$ is of rank at least $2^n$ in the maximum non-trivial $A_{S^2}$ grading with equality if and only if $L$ is a spherical braid closure.
 \end{restatable}

Here a homologically non-trivial link is a link $L$ for which $[L]\neq 0\in H_1(S^1\times S^2)$. Note that a homologically non-trivial link may have homologically trivial sublinks. $|L\cap (\{*\}\times S^2)|$ indicates the minimal geometric intersection number of $L$ with $\{*\}\times S^2$. We defer the definition of the grading $A_{S^2}$ to Section~\ref{sec:review}. This result is a version of Martin’s braid detection result for link Floer homology~\cite[Proposition
1.1]{martin2022khovanov}. Note that, as a consequence of~\cite[Proposition 9.18]{juhasz2006holomorphic} and~\cite[Theorem 1.4]{juhasz2008floer}, $|L\cap (\{*\}\times S^2)|=1$ if and only if $\rank(\widehat{\HFK}(S^1\times S^2,L))=0$, in which case $L$ is the closure of the spherical $1$ braid, i.e. the core of $0$ surgery on the unknot in $S^3$. The hypothesis that $L$ has irreducible exterior in the statement of Theorem~\ref{thm:sphericalbraiddetection} is again readily removed by appropriate applications of the K\"unneth formula, Equation~(\ref{eq:Kunneth}).

 This paper is organized as follows. In Section~\ref{sec:review} we review sutured manifolds and sutured Floer homology and prove some preparatory lemmas. Section~\ref{sec:mainthm} is devoted to the proof of the main theorem. In Section~\ref{sec:S1S2} we study the link Floer homology of links in $S^1\times S^2$. 
\subsection*{Acknowledgments}  The author would like to offer a special thanks to Yi Ni. Ni's mathscinet review for Li-Xie-Zhang's paper~\cite{li2022floer} served as an initial inspiration for this project, and he also provided the author useful feedback on earlier drafts of this paper. He is also grateful to Subhankar Dey, Sudipta Ghosh, Boyu Zhang and Claudius Zibrowius for various other useful conversations on related topics. Finally he would like to thank the referee for their close reading of the paper.

\section{Sutured Manifolds and Floer homology}\label{sec:review}
We begin by providing some background on sutured manifolds, in Section~\ref{subsec:suturedmnflds}, and sutured Floer homology, in Section~\ref{subsec:SFH}. We also prove a few preparatory results. Specifically, we show the existence of certain sutured hierarchies in Theorem~\ref{thm:hierarchy}, although we note this result essentially amount to a rephrasing of~\cite[Theorem 2.9]{li2022floer} that we find more convenient for our purposes. This will be used in the proof of Theorem~\ref{thm:sphericalbraiddetection} and Theorem~\ref{thm:main}.
\subsection{Sutured Manifolds}\label{subsec:suturedmnflds} Sutured manifolds were first introduced by Gabai to study taut foliations~\cite{gabai1983foliations}. A \emph{sutured manifold} $(Y,\gamma)$ is an oriented manifold $Y$ with non-empty boundary equipped with a decomposition of $\partial Y$ into three pieces $R_+(\gamma)$, $R_-(\gamma)$ and $\gamma$ where, for us, $\gamma$ is a disjoint union of annuli. These pieces are required to satisfy some orientability conditions --- see~\cite[Definition 2.6]{gabai1983foliations}. The union of the cores of the annuli $\gamma$ will be denoted by $s(\gamma)$.

Examples of sutured manifolds include $(Y(K),\mu_K)$, the sutured manifold with $Y(K)$ given by the exterior of a knot $K$ in a closed $3$-manifold $Y$, and $s(\mu_K)$ given by a pair of appropriately oriented meridians of $K$. More generally, given a sutured manifold $(Y,\gamma)$ and a knot $K$ in the interior of $Y$, one can construct a sutured manifold by removing a tubular neighborhood of $K$ from $Y$ and endowing the new boundary component with a pair of parallel oppositely oriented meridians for $s(\gamma)$. We will denote this sutured manifold by $(Y(K),\gamma(K))$. This has a special boundary component which we shall denote by $\partial_K(Y(K))$ or simply $\partial_K$, namely the boundary of a tubular neighborhood of $K$. Likewise, we shall denote the union of the two meridional sutures on $\partial_K(Y(K))$ by $\mu_K$.

We will be particularly interested in sutured manifolds that satisfy the following definition.

\begin{definition}
A sutured manifold is \emph{balanced} if $\chi(R_+(\gamma))=\chi(R_-(\gamma))$ and every component of $\partial Y$ contains at least one component of $\gamma$.
    \end{definition}

Indeed, to apply certain results from sutured Floer homology theory, we will have to work with sutured manifolds satisfying an even stronger condition:

\begin{definition}
     A sutured manifold $(Y,\gamma)$ is \emph{strongly balanced} if for every component $F$ of $\partial Y$, $${\chi(F\cap R_+(\gamma))=\chi(F\cap R_{-}(\gamma))}.$$
\end{definition}

\begin{remark}\label{rem:stronger}
Given a balanced sutured manifold, $(Y,\gamma)$, one can obtain a strongly balanced sutured manifold by picking an appropriate collection of points on the $\gamma$ and gluing \emph{product $1$-handles} to neighborhoods of pairs of points. See~\cite[Figure 8]{kronheimer2010knots}. We will use this technique in the proof of Theorem~\ref{thm:main}.
\end{remark}

\begin{definition}[{\cite[Definition 2.10]{gabai1983foliations}}]
    A sutured manifolds is \emph{taut} if $(Y,\gamma)$ is \emph{taut} if:\begin{enumerate}
        \item $Y$ is irreducible,
        \item $R_\pm(\gamma)$ are incompressible,
        \item $R_\pm(\gamma)$ are Thurston norm minimizing.
    \end{enumerate}
\end{definition}

\begin{definition}
    Let $(Y,\gamma)$ be a balanced sutured manifold. If $\Sigma$ is a connected properly embedded surface in $(Y,\gamma)$ define the \emph{sutured Thurston norm of $\Sigma$} by \begin{equation*}
        x(\Sigma):=\max\Bigg\{\frac{|\partial \Sigma\cap s(\gamma)|}{2}-\chi(\Sigma)),0\Bigg\}.
    \end{equation*}

    Extend this to disconnected surfaces linearly. For a homology class $\alpha\in H_2(M,\partial M)$ define the \emph{sutured Thurston norm} of $\alpha$ by \begin{equation*}
        x(\alpha):=\min\{x(\Sigma):\Sigma\text{ is properly embedded in }Y\text{ and }[\Sigma]=\alpha\}.
    \end{equation*}
\end{definition}

The only result concerning sutured manifolds that we will use is our proof of Theorem~\ref{thm:main} is a slight reinterpretation of a result of Li-Xie-Zhang~\cite[Theorem 2.9]{li2022floer}, which is, in turn, an upgrade of a result of Scharlemann~\cite[Theorem 4.19]{scharlemann1989sutured} to a setting more amenable to results in sutured Floer homology.

We begin with a lemma in singular homology. When using singular homology, we will take coefficients in $\Z$ unless otherwise specified.

\begin{lemma}\label{lem:sinhom}
    Let $c$ be a non-trivial class in $H_{1}(Y)$ where $Y$ is a $3$-manifold. If $c\cap\alpha\neq 0$ for all non-trivial $\alpha\in H_2(Y,\partial Y)$ then $\rank(H_2(Y,\partial Y))=1$. In particular, if $Y$ contains a toroidal boundary component then $\partial Y$ consists of that torus together with a collection of spheres.

\end{lemma}

\begin{proof}
  Let $c$, $Y$ be as in the statement of the lemma. Suppose towards a contradiction that ${\rank(H_2(Y,\partial Y))\geq 1}$. Let $\alpha,\beta$ be linearly independent classes in $H_2(Y,\partial Y)$. Then $\alpha\cap c=k$, $\beta\cap c=l$ with $k,l\neq 0$. Then $(l\alpha-k\beta)\cap \gamma=0$ so that $l\alpha-k\beta=0$ by assumption, contradicting the fact that $\alpha$ and $\beta$ are linearly independent. Thus, $\rank(H_2(Y,\partial Y))\leq 1$. Now observe that $c\cap\PD(c)\neq 0$, so that $\PD(c)\neq 0$, thus $\rank(H_2(Y,\partial Y))=1$ as desired.

For the final claim in the statement of the Lemma, observe that the universal coefficient theorem for homology and the half-lives half-dies principle implies that: $$\rank(H_1(Y;\Z))=\rank(H_1(Y;\Q))\geq \dfrac{1}{2}\rank(H_1(\partial Y;\Q)).$$ By Poincar\'e duality and the universal coefficient theorem, $\rank(H_2(Y,\partial Y))=\rank(H^1(Y))=\rank (H_1(Y))$. If $\partial Y$ has a torus component, then the inequality $1=\rank(H_1(Y))\geq \rank(H_1(\partial Y;\Q))$ implies that the remaining boundary components $C$ have $H_1(C;\Q)=0$ --- i.e. they are spheres.
\end{proof}

For the statement of the next theorem, recall that if $\Sigma$ is a \emph{decomposing surface} in a sutured manifold $Y$ subject to certain conditions, there is a procedure called a \emph{sutured manifold decomposition} for obtaining a sutured manifold structure on the $Y\setminus\nu(\Sigma)$, where $\nu(\Sigma)$ is a small neighborhood of $\Sigma$ in $Y$. For details, see~\cite[Definition 3.1]{gabai1983foliations}. If $(Y',\gamma')$ is obtained from $(Y,\gamma)$ by a sutured manifold decomposition along a surface $\Sigma$ then we write $(Y,\gamma)\overset{\Sigma}{\rightsquigarrow}(Y',\gamma')$.

A decomposing surface $\Sigma$ in a sutured manifold $(Y,\gamma)$ is called \emph{taut} if the sutured manifold obtained by decomposing $(Y,\gamma)$ along $\Sigma$ is taut. Examples of taut decomposing surfaces include the co-cores of the $1$-handles discussed in Remark~\ref{rem:stronger}, provided that the initial sutured manifold is taut.

\begin{theorem}\label{thm:hierarchy}
 Let $(Y,\gamma)$ be a balanced sutured manifold, $K$ be a non-trivial knot in $Y$. If $(Y(K),\gamma(K))$ is taut then either $(Y(K),\gamma(K))$ contains a product disk or there is a sequence of sutured manifold decompositions:
\begin{equation*}    
   (Y(K),\gamma(K)) \overset{\Sigma_1}\rightsquigarrow (Y_1(K),\gamma_1(K)) \overset{\Sigma_1}\rightsquigarrow\dots\overset{\Sigma_{n-1}}\rightsquigarrow(Y_{n-1}(K),\gamma_{n-1}(K))\overset{\Sigma_{n}}{\rightsquigarrow} (Y_n,\gamma_n)
\end{equation*}

\noindent such that:
\begin{enumerate}
\item Each $\Sigma_i$ is either a product disk or a surface representing any chosen class $\alpha\in H_2(Y,\partial Y)$ with $\partial_*(\alpha)\neq 0$ in $H_1(\partial Y)$.
\item For $i\neq n$ we have that $\Sigma_i$ does not intersect $\partial_KY(K)$.
\item $\Sigma_n\cap\partial_KY(K)\neq\emptyset$ but $\Sigma_n\cap\mu_K=\emptyset$.
\item $(Y_i(K),\gamma_i(K))$ is taut for all $i$, as is $(Y_n,\gamma_n)$.
\item\label{pnt:boundaryproperties} For every component $V$ of $R(\gamma_i(K))$, $\Sigma_i\cap V$ consists of parallel oriented boundary coherent simple closed curves (c.f. ~\cite[Definition 1.2]{juhasz2008floer}).

\end{enumerate}

\end{theorem}

Note that we are abusing notation here by allowing $K$ to refer to knots in different sutured manifolds. 
   We follow the proof of~\cite[Theorem 4.19]{scharlemann1989sutured}, making the adaptions outlined in the proof of~\cite[Theorem 2.9]{li2022floer}, to ensure that the relevant surfaces and sutured manifolds satisfy the conditions necessary to apply Juh\'asz' sutured Floer homology decomposition formula~\cite[Theorem 1.3]{juhasz2008floer} at each stage of the sutured manifold hierarchy. Since the proof is essentially that given for~\cite[Theorem 2.9]{li2022floer} --- or indeed for~\cite[Theorem 4.19]{scharlemann1989sutured} --- we only provide a sketch.
\begin{proof}[Proof Sketch]

Suppose $(Y,\gamma)$ and $K$ are as in the statement of the theorem. We construct the surfaces $\Sigma_i$ iteratively, at each stage reducing the \emph{complexity}, $C(Y,\gamma)$, or \emph{reduced complexity}, $\widehat{C}(Y,\gamma)$, of $(Y_i,\gamma_i)$. See~\cite[Definition 4.11, Definition 4.12]{scharlemann1989sutured} for the definitions of these two terms. 

The first step in Scharlemann's proof of~\cite[Theorem 4.19]{scharlemann1989sutured} is to decompose along index $0$ disks in $(Y,\gamma)$. See~\cite[Definition 4.6]{scharlemann1989sutured} for a definition of ``index" in this context. There are three types of index zero disks, see the remark after~\cite[Definition 4.6]{scharlemann1989sutured}. Since we are working with a sutured manifold rather than a sutured manifold with a properly imbedded $1$-complex, canceling and amalgamating disks cannot occur (See~\cite[Definition 4.1]{scharlemann1989sutured}). Thus we only have product disks.

Suppose that there are product disks embedded in $(Y(K),\gamma(K))$. Then we can decompose along product disks until we obtain a sutured manifold $(Y_i(K),\gamma_i(K))$ with no more product disks, since each product disk decomposition reduces the complexity. Note that none of these product disks can have boundary contained in $\partial_KY(K)$, since we are assuming $K$ is non-trivial.

 Having obtained a sutured manifold without product disks, $(Y_i(K),\gamma_i(K))$, we proceed as follows. If there exists a class $\alpha\in H_2(Y_i(K),\gamma_i(K))$ such that $\alpha\cap[\mu_K]=0$ and $\partial_*(\alpha)\neq 0$ in $H_1(\partial Y)$ then we can pick a properly embedded surface $\Sigma_i$ in $(Y_i,\gamma_i)$ with $[\Sigma_i]=\alpha$ such that the sutured manifold obtained by decomposing $(Y_i,\gamma_i)$ along $\Sigma_i$ is taut. The conditions on $\Sigma_i$ described in~\ref{pnt:boundaryproperties} can be achieved by modifying $\Sigma_i$ near its boundary as described in the proof of~\cite[Theorem 2.9]{li2022floer}.

We apply the decompositions described above recursively. At each stage, the complexity of $(Y_i,\gamma_i)$ strictly decreases. Since the complexity is bounded below, this process must terminate. There are two ways in which this can occur; either some class $\alpha\in H_2(Y,\partial Y)$ with $\alpha\cap[\mu_K]=0$ and and $\partial_*(\alpha)\neq 0$ is represented by a surface $\Sigma_\alpha$ with $\partial \Sigma_\alpha\cap\partial_KY_K\neq\emptyset$ --- so that we have the desired result --- or at some stage there is no such class $\alpha\in H_2(Y,\partial Y)$ with $\alpha\cap[\mu_K]=0$ and $\partial_*(\alpha)\neq \emptyset$.

It thus remains only to exclude the second case. To do so, observe that by Lemma~\ref{lem:sinhom} the remaining boundary components of the component of $(Y_i,\gamma_i)$ containing $\partial_KY_K$ are all spheres. Observe that there must be at least one sphere, since we are assuming that $(Y,\gamma)$ is a sutured manifold, so that in turn each $(Y_i,\gamma_i)$ must have non-empty boundary. However, we then have that $(Y_i,\gamma_i)$ contains at least one reducing sphere, so that $(Y(K),\gamma(K))$ is not taut, a contradiction.
\end{proof}

\subsection{Sutured Floer homology}\label{subsec:SFH} Sutured Floer homology is an an invariant of balanced sutured manifolds due to Juh\'asz~\cite{juhasz2006holomorphic} that simultaneously generalizes Heegaard Floer homology and knot Floer homology. In this subsection we briefly review some of the properties of sutured Floer homology that we will use in the rest of the paper, and prove some preparatory lemmas. We refer the reader to~\cite{juhasz2006holomorphic,juhasz2010sutured} for further background on sutured Floer homology.

 While sutured Floer homology can be defined with integer coefficients, we take coefficients in $\Z/2$ in this paper for the sake of simplicity. Sutured Floer homology satisfies a K\"unneth formula for connect sums. Let $(Y_i,\gamma_i)$ be connected product sutured manifolds for $1\leq i\leq n$ with $n\geq 1$ and $M_i$ be $3$-manifolds with $1\leq i\leq m$, perhaps with $m=0$. Then~\cite[Proposition 9.15]{juhasz2006holomorphic} implies that:

\begin{equation}\label{eq:Kunneth}
    \SFH\Big(\big(\underset{1\leq i\leq m}{\#}M_i\big)\#\big(\underset{1\leq i\leq n}{\#}(Y_i,\gamma_i)\big)\Big)\cong \Big(\underset{1\leq i\leq m}{\bigotimes}\widehat{\HF}(M_i)\Big)\otimes \Big(\underset{1\leq i\leq n}{\bigotimes}\SFH(Y_i,\gamma_i)\Big)\otimes V^{n-1}.
\end{equation}
Here, and for the remainder of the paper, $V$ is a rank two vector space. The sutured Floer homology of a sutured manifold $(Y,\gamma)$ can be endowed with a grading by relative $\spin^c$ structures on $Y$ or a relative grading by $H^2(Y,\partial Y)\cong H_1(Y)$ grading. The collection of relative $\spin^c$ structures is in affine correspondence with $H^2(Y,\partial Y)$. A choice of trivialization, $t$, of the plane field orthogonal of the $\spin^c$ structure restricted to $\partial Y$ and an element $\alpha\in H_2(Y,\partial Y)$ allows one to extract a $\Q$-valued grading via the map from ${A_{\alpha,t}:\spin^c(Y,\gamma)\to\Q}$ given by $\s\mapsto \langle c_1(\s,t),\alpha\rangle$. The fact that appropriate trivializations exist for strongly balanced sutured manifolds was proven by Juh\'asz~\cite[Propositon 3.4]{juhasz2008floer}. For the boundary components $\partial_K$ of $Y(K)$ we will always use the trivialization specified by the vector field pointing in the meridional direction. We will typically suppress the choice of trivialization in our notation because it does not play a significant role in this paper. While the $A_{\alpha}$ grading is $\Q$ valued apriori, note that if $\s$ and $\s'$ are two $\spin^c$ structures, then $A_\alpha(\s)-A_\alpha(\s')\in 2\Z$, since $c_1(\s,t)-c_1(\s',t)=2(\s-\s')\in H^2(Y,\partial Y)$. In the special case of a knot exterior, this grading is exactly twice the Alexander grading. Observe also that a graded version of the K\"unneth theorem --- Equation~(\ref{eq:Kunneth}) holds --- this follows from the fact that the $H_2(X\#Y,\partial(X\#Y))\cong H_2(X,\partial X)\oplus H_2(Y,\partial Y)$ for $X$ and $Y$ $3$-manifolds and $H^1(X\#Y)\cong H^1(X,\partial X)\oplus H^1(Y,\partial Y)$.

If $K$ is a knot in a sutured manifold $(Y,\gamma)$ then there is a filling map $${G_K:\spin^c(Y(K),\gamma(K))\to\spin^c(Y,\gamma)},$$ This is obtained by gluing $\nu(K)$ to $Y(K)$, where $\nu(K)$ is equipped with meridional sutures, $(\nu(K),\mu_K)$ --- i.e. the exterior of the core of $0$ surgery on the unknot equipped with meridional sutures --- endowed with a non-vanishing vector field with the homology type of the torsion relative $\spin^c$-structure on $(D^2\times S^1,\mu)$. See Figure~\ref{fig:solidtorusvfields} and~\cite[Section 3.7]{HolomorphicdiskslinkinvariantsandthemultivariableAlexanderpolynomial} for a version of this map in a slightly different context. From a trivialization $t$ for $(Y(K),\gamma(K))$, one obtains a trivialization $\overline{t}$ for $(Y,\gamma)$ simply by forgetting the component of the trivialization on $\partial_K(Y(K))$. Indeed, $t$ can be extended over $\nu(K)$ with vanishing set $K$, so given a generic zero section of the plane field orthogonal to $\s\in \spin^c(Y(K),\gamma(K))$ that agrees with $t$ on $\partial_K$, one can obtain a zero section of the two plane field orthogonal to $G_K(\s)\in\spin^c(Y,\gamma)$ by adding $K$.  Finally, there is a natural map $H_K: H_2(Y,\partial Y)\to H_2(Y(K),\partial(Y(K)))$ which takes $[\Sigma]$ where $\Sigma$ is a properly embedded surface in $Y$ that meets $K$ transversely, to $[\Sigma\setminus\nu(K)]$. The next lemma is a modest generalization of~\cite[Lemma 3.13]{HolomorphicdiskslinkinvariantsandthemultivariableAlexanderpolynomial}, and implies the rank bound~\ref{eq:rankbound2} stated in the introduction. Its proof is partially based on the proofs of Proposition 7.1 and Lemma 3.13 in~\cite{HolomorphicdiskslinkinvariantsandthemultivariableAlexanderpolynomial}.

\begin{figure}
    \centering
     \begin{tikzpicture}
        % Outer circle (annulus boundary)
        \draw[thick] (0,0) circle (3);

        % Inner circle (annulus boundary)
        \draw[thick, blue] (0,0) circle (1);

         \draw[thick, blue, ->] (1,0) -- ++(0,-0.01);
        
        % Leftward arrow at 180 degrees
        \draw[thick, blue, ->] (-1,0) -- ++(0,-0.01);

        % Vertical flow lines curving AWAY from the inner boundary
        \draw[thick, blue,->] (-2.5,1.65) to[out=-90, in=90] (-2.75,0);
      \draw[thick, blue] (-2.75,0)  to[out=-90, in=90] (-2.5,-1.65);
      
        \draw[thick, blue,->] (-1.5,2.6) to[out=-90, in=90] (-2,0);
        \draw[thick, blue](-2,0) to[out=-90, in=90] (-1.5,-2.6);
        
        \draw[thick, blue,->] (0,3) to (0,2);
         \draw[thick, blue] (0,3) to (0,1);
         \draw[thick, blue,->] (0,-1) to (0,-2);
           \draw[thick, blue] (0,-1) to (0,-3);
      \draw[thick, blue,->] (2.5,1.65) to[out=-90, in=90] (2.75,0);
      \draw[thick, blue] (2.5,1.65) to[out=-90, in=90] (2.75,0) to[out=-90, in=90] (2.5,-1.65);
        \draw[thick, blue] (1.5,2.6) to[out=-90, in=90] (2,0) to[out=-90, in=90] (1.5,-2.6);
        \draw[thick, blue,->] (1.5,2.6) to[out=-90, in=90] (2,0);
    \end{tikzpicture}
    \caption{The torsion relative $\spin^c$-structure on $(D^2\times S^1,\mu)$ can be constructed from the vector field $v$, shown in blue, on the annulus. Here the core of $D^2\times S^1$, $K$, is the inner boundary component of the annulus and oriented counterclockwise. The section of $v^\perp$, restricted to the annulus, points perpendicularly out of the page and decreases in magnitude to $0$ as you move from the outer boundary to the inner boundary. The vector field $v$ on the solid torus can be recovered by rotating the figure about the inner boundary component and modifying the resulting vector field in a neighborhood of the right hand side of the inner boundary component. The section of $v^\perp$ can be defined similarly.}
    \label{fig:solidtorusvfields}
\end{figure}
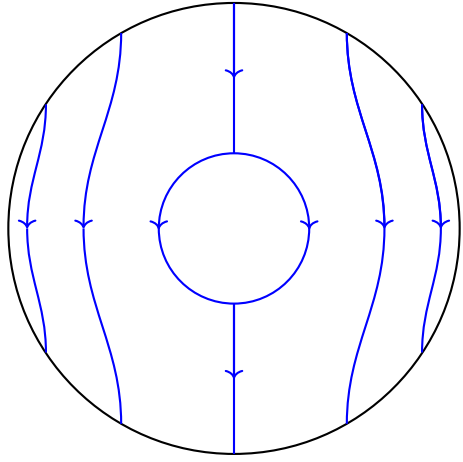

\begin{lemma}\label{lem:spectral}
    Suppose $K$ is a knot in a stongly balanced sutured manifold $(Y,\gamma)$ such that $[\mu_K]$ is non-torsion in $ H_1(Y(K);\Z)$ and let $\alpha$ be a class in  $H_2(Y,\partial Y)$. View $\SFH(Y,\gamma)$ as an $A_\alpha$ graded vector space and $\SFH(Y(K),\gamma(K))$ as a $A_{H_K(\alpha)}$-graded vector space. There is a spectral sequence from $\SFH(Y(K),\gamma(K))$ to $\SFH(Y,\gamma)\otimes V$. Moreover 
  for any $C\in\Z$;
    \begin{align}\rank(\SFH(Y(K),\gamma(K),A_{[H_K(\alpha)]}=C))\geq2\rank(\SFH(Y,\gamma),A_{[\alpha]}=C+ \langle \PD([K]),\alpha\rangle)).
    \end{align}

\end{lemma}

Note that the hypothesis that $[\mu_K]$ is non-torsion in $H_1(Y(K))$ is necessary; the core of $0$ surgery on the unknot in $S^3$ has trivial sutured Floer homology, but $0$-surgery on the unknot in $S^3$ --- i.e. $S^1\times S^2$ --- has non-trivial sutured Floer homology.

\begin{proof}
    For the ungraded statement, consider a Heegaard diagram for $(Y(K),\gamma(K))$ $(\Sigma,\alpha,\beta)$. Let $(\mathring{Y},\mathring{\gamma})$ be the sutured manifold obtained from $(Y,\gamma)$ by removing an embedded $3$-ball and adding a single suture to the corresponding boundary component. A Heegaard diagram for $(\mathring{Y},\mathring{\gamma})$ can be obtained from $(\Sigma,\alpha,\beta)$ by capping off a single boundary component corresponding to a suture on $\partial_K(Y(K))$. Equivalently, $\SFH(Y(K),\gamma(K))$ and $\SFH(\mathring{Y},\mathring{\gamma})$ can each be computed as follows. First, given a Heegaard diagram $(\Sigma,\alpha,\beta)$ for $(Y,\gamma)$ collapse a boundary component of $\Sigma$ corresponding to $\mu_K$ to a point $z$ in a new pointed Heegaard diagram $(\overline{\Sigma},\alpha,\beta,z)$. Observe that to compute $\SFH(Y(K),\gamma(K))$ one can count pseudo-holomorphic disks which have multiplicity zero at $z$. This yields a differential $\partial_{\CF(Y(K),\gamma(K))}$ on $\CF(Y(K),\gamma(K))$ which computes $\SFH(Y(K),\gamma(K))$. To compute $\SFH(\mathring{Y},\mathring{\gamma})$, one can count pseudo-holomorphic disks with arbitrary multiplicity at $z$. This yields a differential $\partial_{\CF(\mathring{Y},\mathring{K})}$ on $\CF(Y(K),\gamma(K))$ which computes $\SFH(\mathring{Y},\mathring{\gamma})$. Consider the decomposition of $\CF(\mathring{Y}(K),\mathring{\gamma}(K))$ as a relatively graded complex over $\langle \PD([\mu_K])\rangle\subset H^2(Y(K),\partial(Y(K)))$.  As in~\cite[Lemma 3.11]{HolomorphicdiskslinkinvariantsandthemultivariableAlexanderpolynomial}, if there is a pseudo-holomorphic disk from $\mathbf{x}$ to $\mathbf{y}\in\CF(Y(K),\gamma(K))$ then the relative  $H_1(Y(K))$ grading of the two generators is given by $n_z(\phi)\cdot[\mu_K]$, where $n_z$ is the multiplicity of the disk at $z$. Since $[\mu_K]$ is non torsion in $H_1(Y(K);\Z)$ it follows from Poincar\'e duality that $\PD([\mu_K])$ is non-torsion in $H^2(Y(K);\Z)$. Therefore, since $n_z$ is non-negative, the differential $\partial_{\CF(\mathring{Y},\mathring{\gamma})}$ is filtered with respect to the relative $\langle \mu_K\rangle$-grading. Consequently, there is a spectral sequence from $\SFH(Y,\gamma)$ to $\SFH(\mathring{Y},\mathring{\gamma})$. But $\SFH(\mathring{Y},\mathring{\gamma})\cong\SFH(Y,\gamma)\otimes V$ by~\cite[Equation 9.1]{juhasz2006holomorphic}, so the ungraded result follows.

      To upgrade to the graded version of this statement, observe that --- again using the Heegaard diagram $(\overline{\Sigma},\alpha,\beta,z)$ to identify generators of  $\CF(\mathring{Y},\mathring{\gamma})$ with generators of  $\CF({Y}(K),\gamma(K))$ --- the relative $A_{\alpha}$-grading on $\CF(\mathring{Y},\mathring{\gamma})$ agrees with the relative $A_{H_K(\alpha)}$-grading on $\CF({Y}(K),\gamma(K))$. Thus, we need to compare $\langle c_1(\s,t),H_K(\alpha)\rangle$ and $\langle c_1(G_K(\s),\overline{t}),\alpha\rangle$. Represent $\alpha$ by a properly embedded surface $\Sigma$ which meets $K$ transversely. Let $\Sigma'$ denote $\Sigma\setminus \nu(K)$ viewed as a properly embedded surface in  $Y(K)$. Observe that if $c_1(\s,t)=\PD[c]$ for some multi-curve $c\subset Y(K)$ then $c_1(G_K(\s),\overline{t})$ is represented by $c\cup K\subset Y$ as discussed prior to this lemma.
    
    Thus 
    \begin{align*}
        \langle( c_1(G_K(\s),\overline{t}),\alpha\rangle&=\cap(c\cup K,\Sigma)\\&=\cap(c,\Sigma)+\cap(K,\Sigma)\\&=\cap(c,\Sigma')+\cap(K,\Sigma)\\&=\langle (c_1(\s),t),H_K(\alpha)\rangle+ \langle \PD([K]),\alpha\rangle.
    \end{align*}

  Therefore, the rank bound descends to the individual gradings as described.\end{proof}

In particular, it follows from Lemma~\ref{lem:spectral} that the $A_{H_K(\alpha)}$ span of $\SFH(Y(K),\gamma(K))$ is at least the $A_{\alpha}$-span of $\SFH(Y,\gamma)$. The proof of Theorem~\ref{thm:main} amounts to determining when these two spans are in fact equal.

We now discuss the relationship between sutured Floer homology and sutured manifold decompositions. First recall that Juh\'asz showed that the sutured Floer homology of a manifold remains unchanged under decomposition along product disks~\cite[Lemma 9.113]{juhasz2006holomorphic}, or equivalently under product $1$-handle attachment. More generally, Juh\'asz showed that if $(Y,\gamma)\overset{\Sigma}{\rightsquigarrow} (Y',\gamma')$ is a sutured manifold condition subject to some conditions, then $\SFH(Y',\gamma')$ is isomorphic to the summand of $\SFH(Y,\gamma)$ supported in the subset of relative $\spin^c$-structures on $(Y,\gamma)$ that are~\emph{outer} with respect to $\Sigma$ ~\cite[Theorem 1.3]{juhasz2008floer}. For a definition of outer, see See~\cite[Definition 1.1]{juhasz2010sutured}. Recall from~\cite[Definition 3.8]{juhasz2008floer}, there are quantities $c(\Sigma,t):=\chi(\Sigma)+I(\Sigma)-r(\Sigma,t)$ one can associate to a properly embedded surface $\Sigma$ in a sutured manifold $(Y,\gamma)$ equipped with a trivialization. If $\Sigma$ is a decomposing surface in $(Y,\gamma)$, then this quantity picks out the relative $\spin^c$ structures which are outer with respect to $\Sigma$ in the sense that $\s$ is outer with respect to $S$ if and only if $\langle c_1(\s,t),[S_i]\rangle=c(S_i,t)$ for every component $S_i$ of $S$~\cite[Lemma 3.10]{juhasz2010sutured}. We also have the following Lemma:
\begin{lemma}\label{lem:maxagrading}
    Suppose $(Y,\gamma)$ is a taut sutured manifold and $\alpha\in H_2(Y,\partial Y)$ is a homology class with a taut representative $S$. The minimum $A_\alpha$-grading in which $\SFH(Y,\gamma)$ is non-trivial is given by the maximum value of $c(\Sigma,t)$ over all decomposing surfaces $\Sigma$ representing $\alpha$ without either disk components whose boundary does not intersect $\gamma$ or spherical components. This maxima is realized by $c(\Sigma',t)$ for any taut representative $\Sigma'$ of $\alpha$.
\end{lemma}

\begin{proof}
    Let $m$ denote the minimal $A_\alpha$ grading in which $\SFH(Y,\gamma)\neq 0$ and $c$ denote the maximal value of $c(\Sigma,t)$ over all taut decomposing surfaces $\Sigma$ representing $\alpha$. Let $S$ be a taut representative of $\alpha$. By definition, decomposing $(Y,\gamma)$ along $S$ results in a taut sutured manifold, which has non-trivial sutured Floer homology by~\cite[Theorem 1.4]{juhasz2008floer}. It follows from~\cite[Theorem 1.3]{juhasz2008floer} that $\SFH(Y,\gamma)$ is non-trivial in the relative $\spin^c$ structures that are outer with respect to $S$. In turn, ~\cite[Lemma 3.10]{juhasz2008floer} implies that $\SFH(Y,\gamma)$ is non-trivial in $A_\alpha$-grading $c$, so that $c\geq m$. 
    
    Now observe that the $g$-fold stabilization of $S$ yields a surface $S^g$ such that decomposing $(Y,\gamma)$ along $S^g$ is not the connect sum of taut sutured manifolds. Thus $\SFH(Y,\gamma)$ is trivial in $A_\alpha$-grading $c(S^g,t)$ by~\cite[Proposition 9.18]{juhasz2006holomorphic}. Computing $c(S^g,t)$ from the definition shows that $c(S^g,t)=c(S,t)-2g$, so we have that $c=m$.

    It remains only to show that if $\Sigma$ is a decomposing surface then $c(\Sigma,t)<c$. But this can argued as above; if $c(\Sigma,t)>c$, then some stabilization $\Sigma'$ of has $c(\Sigma,t)=c$, but then the manifold obtained by decomposing $(Y,\gamma)$ along $\Sigma'$ has non-trivial sutured Floer homology, a contradiction, since $\Sigma'$ is stabilized.
\end{proof}

\section{Floer minimal knots in sutured manifolds}\label{sec:mainthm}

In this section we prove Theorem~\ref{thm:main}, which we restate for the reader's convenience:

\main*

 Our first goal is to reduce to the case that $(Y(K),\gamma(K))$ is taut, so that we can apply Theorem~\ref{thm:hierarchy}. We begin by studying the case that $(Y(K),\gamma(K))$ is compressible.

\begin{lemma}\label{lem:incompressible}
    Suppose $R_\pm(\gamma)$ are incompressible. Then either $(Y(K),\gamma(K))$ is incompressible, or $(Y,\gamma)$ can be written as $(Y',\gamma')\#(S^1\times  S^2)$ and $K$ is the spherical $1$-braid closure in $S^1\times S^2$.
\end{lemma}

\begin{proof}
    Suppose $R_\pm(\gamma(K))$ is compressible. Let $D$ be a compressing disk. Observe that either $\partial D$ is contained in $\partial Y$, so that $(Y,\gamma)$ is compressible, or $\partial D\subset\partial_K(Y(K))$. Observe that $\partial D$ meets $\partial(Y(K))$ in a meridian for $K$. It follows that there is an embedded sphere, $\overline{D}$ in $Y$ obtained by capping off $D$ with a meridional disk for $K$. Moreover $\overline{D}$ is intersected exactly once by $K$. It follows that $\overline{D}$ is non-separating. Consider a neighborhood of $\overline{D}\cup K$, $N$. Observe that $N$ is given by $S^1\times  S^2\setminus B^3$ and that $\partial N$ is a separating $2$-sphere witnessing the decomposition of $(Y(K),\gamma(K))$ as $(Y',\gamma')\#T_0$ and $(Y,\gamma)$ as $(Y',\gamma')\# (S^1\times  S^2)$, where $T_0$ is the unknot with sutures of slope zero.
\end{proof}

\begin{lemma}
    Suppose that $K$ is a knot in a sutured manifold $(Y,\gamma)$ such that $R_\pm(\gamma)$ are Thurston norm minimizing in $H_2(Y,\gamma$ and $R_\pm(\gamma(K))$ are incompressible. Then $R_\pm(\gamma(K))$ are Thurston norm-minimizing in $H_2(Y(K),\gamma(K))$.
\end{lemma}

\begin{proof}
    Suppose that there is a surface $\Sigma$ in $(Y(K),\gamma(K)))$ representing $[R_\pm(\gamma(K))]\in H_2(Y(K),\gamma(K))$ with $$x(\Sigma)< x(R_\pm(\gamma(K)))=x(R_\pm(\gamma)).$$ Observe that $\Sigma$ meets the boundary component of $Y(K)$ corresponding to $K$ in a collection of meridians. Thus in $Y$ we can obtain a surface $\overline{\Sigma}$ in $Y$ from $\Sigma$ by capping off the corresponding boundary components of $\Sigma$ with disks. Observe that $x(\overline{\Sigma})\leq x(\Sigma)\leq x(R_\pm(\gamma))$ Since $R_\pm(\gamma)$ are Thurston norm minimizing for $[\overline{\Sigma}]=[R_\pm(\Sigma)]$ it follows that $x(\overline{\Sigma})=x(\Sigma)$. This implies that every component of $\Sigma$ that has a disk added has non-negative Euler characteristic; i.e. such components are disks or annuli. But if $\Sigma$ contains such an annular component, then $\overline{\Sigma}$ would contain a compressing disk for $R_\pm(\gamma)$. On the other hand, if $\Sigma$ contains a disk component, then $R_\pm(\gamma(K))$ would be compressible, a contradiction.
\end{proof}

Before the next lemma we fix some notation. If $Y$ is a manifold with a spherical boundary component, we let $\overline{Y}$ be the manifold obtained by capping off a fixed spherical boundary component with a $3$-ball. We leave the specific spherical boundary component implicit on our notation. If $Y$ carries a sutured structure $(Y,\gamma)$ then set $(\overline{Y},\overline{\gamma})$ to be the sutured structure on $\overline{\gamma}$ given by all components of $\gamma$ that do not lie on the spherical boundary of $Y$ capped off with a $3$-ball.

\begin{lemma}\label{lem:spericalboundarycase}
    Suppose $K$ is a knot in a sutured manifold $(Y,\gamma)$ with spherical boundary components. If $\SFH(Y(K),\gamma(K))=2\SFH(Y,\gamma)$  and $|\partial Y|>1$ then $\SFH(\overline{Y}(K),\overline{\gamma}(K))=2\SFH(\overline{Y},\overline{\gamma})$. If $|\partial Y|=1$ then $\SFH(\overline{Y}(K),\gamma(K))=\widehat{\HF}(\overline{Y})$.
\end{lemma}
\begin{proof}
    Recall that if $(\mathring{Y},\mathring{\gamma})$ is obtained from $(Y,\gamma)$ by removing an embedded $3$-ball and endowing the resulting spherical boundary component with a single suture, then $\SFH(\mathring{Y},\mathring{\gamma})\cong\SFH(Y,\gamma)\otimes V$. Moreover, if $(Y',\gamma')$ is obtained from $(Y,\gamma)$ by adding an extra pair of appropriately oriented parallel sutures then $\SFH(Y',\gamma')\cong\SFH(Y,\gamma)\otimes V$. Finally, if $(Y,\gamma)$ is obtained from a closed manifold $M$ by removing a ball and endowing the resulting spherical boundary component with a single suture, then $\widehat{\HF}(M)\cong\SFH(Y,\gamma)$. The result follows directly from these three facts.
\end{proof}

The above lemma reduces the problem to treating the case in which $\partial Y$ has no spherical boundary components. Our next goal is to understand spheres embedded in the interior of $Y$. In that direction, we state a version of Knesser's theorem for sutured manifolds. Recall that given a group $G$ one can define $p(G)$ to be the infimum of the cardinality of a generating set for $G$. If $G$ is finitely presented and is the free product of two groups $G_1$ and $G_2$ then $p(G)=p(G_1)+p(G_2)$ by Grushko's theorem.
\begin{lemma}\label{lem:knesser}
    Any sutured manifold without spherical boundary components can be written as a finite connect sum of sutured manifolds $(Y_i,\gamma_i)$ and closed $3$-manifolds $M_j$ where each $Y_i$ is irreducible and each $M_j$ is either irreducible or $S^2\times S^1$.
\end{lemma}

We prove this for the sake of completeness.

\begin{proof}
  We will proceed by induction on the complexity of the sutured manifold, which we define to be $p(\pi_1(Y))$ and the number of boundary components of $(Y,\gamma)$. In the base case the sutured manifold has one boundary component and the fundamental group is trivial. By the half-lives half dies-principle it follows that $\rank(H_1(\partial Y))\leq\frac{1}{2}\rank(H_1(Y))\leq p(\pi_1(Y))=0$, so that $\partial Y$ must be a disjoint union of spheres, contradicting our assumption.

  For the inductive step suppose $S$ is an embedded reducing sphere in $(Y,\gamma)$. If $S$ is one sided, then we can pick a path $\phi$ from $S$ to $S$ such that the boundary of a neighborhood is a splitting sphere realizing $Y$ as the connect sum of a sutured manifold $(Y',\gamma')$ and $S^1\times  S^2$. Since $$p(\pi_1(Y'))=p(\pi_1(Y))+p(\pi_1(S^1\times S^2))=p(\pi_1(Y))+1,$$ we can proceed by induction.

  If $S$ is two-sided, then we can realize $(Y,\gamma)$ as either the connect sum of two sutured manifolds or the connect sum of a sutured manifold and a closed manifold. In the former case, we are done because the two resulting sutured manifolds have a smaller number of boundary components. In the latter case $\pi_1(Y)\cong\pi_1(Y_1)*\pi_1(Y_2)$. Thus, unless $\pi_1(Y_i)=\{*\}$ for either $i$ we  can proceed by induction. If, say, $\pi_1(Y_1)=\{*\}$, then once again the half-lives half-dies principle implies that all of $Y_1$'s boundary components are spherical. Thus, $Y_1$ has no boundary components by assumption. The Poincar\'e conjecture then implies that $Y_i=S^3$, contradicting our assumption that $S$ is a reducing sphere.
\end{proof}

We can now prove the main theorem.
\begin{proof}[Proof of Theorem~\ref{thm:main}]

  Suppose $(Y,\gamma)$, $K$ are as in the statement of the Lemma. By Lemma~\ref{lem:spericalboundarycase}, it suffices to treat the case in which $\partial Y$ has no spherical components.
  
  We first reduce to the case that $(Y(K),\gamma(K))$ is irreducible. By Lemma~\ref{lem:knesser} we can write $(Y(K),\gamma(K))$ as a connected sum:
   \begin{equation*}
      \big(\underset{1\leq i\leq m}{\#}M_i\big)\#\big(\underset{1\leq j\leq n}{\#}(Y_i,\gamma_i)\big)\end{equation*} 
      
\noindent      where each $Y_j$ is irreducible and each $M_i$ is either irreducible or $S^1\times S^2$. Observe that each $(Y_i,\gamma_i)$ is taut, as else $\rank(\SFH(Y_i,\gamma_i))=0$~\cite[Proposition 9.18]{juhasz2006holomorphic} so that $\rank(\SFH(Y(K),\gamma(K)))=0$ by the generalized K\"unneth formula for sutured Floer homology, a contradiction.  We have two cases; either one of the manifolds $Y_j$ has a single boundary component $\partial_K Y$ corresponding to $K$ or there is no such component.

   Suppose we are in the first case. After relabeling we may take $j=1$. Then $(Y_1,\gamma_1)$ is given by the sutured exterior of a knot in the three manifold $N$ obtained by capping of a meridional suture for $(Y_j,\gamma_j)$ with a disk and then filling in the resulting spherical boundary component with a $3$-ball. Observe that:

   \begin{equation*}   2\rank(\SFH(N\#\big(\underset{1\leq i\leq m}{\#}M_i\big)\#\big(\underset{2\leq j\leq n}{\#}(Y_i,\gamma_i)\big)=\rank(\SFH(\big(\underset{1\leq i\leq m}{\#}M_i\big)\#\big(\underset{1\leq j\leq n}{\#}(Y_i,\gamma_i)\big))
   \end{equation*}

  So that $2\rank(\widehat{\HF}(N))=\rank(\SFH(N_K,\gamma_K))=\rank(\widehat{\HFK}(K,N))$ by two applications of Equation~(\ref{eq:Kunneth}).
 Thus $K$ is a Floer simple knot in $N$, and we have the desired result.
 
 The rest of this proof is devoted to treating the second case. After relabeling, we may take $\partial_K Y$ to be a boundary component of $Y_1$. Let $(Y_1',\gamma_1')$ be the sutured manifold obtained by capping off a meridional suture for $K$ and filling in the spherical boundary component. Observe that:
  \begin{equation*}   2\rank(\SFH((Y_1',\gamma_1')\#\big(\underset{1\leq i\leq m}{\#}M_i\big)\#\big(\underset{2\leq j\leq n}{\#}(Y_i,\gamma_i)\big)=\rank(\SFH(\big(\underset{1\leq i\leq m}{\#}M_i\big)\#\big(\underset{1\leq j\leq n}{\#}(Y_i,\gamma_i)\big))
   \end{equation*}
 
\noindent so that $2\rank(\SFH(Y_1',\gamma_1'))=\rank(\SFH(Y_1,\gamma_1))$ by two applications of Equation~(\ref{eq:Kunneth}). Thus, it is enough to treat the case that $(Y(K),\gamma(K))$ is irreducible and has at least two boundary components. We duly assume that $(Y(K),\gamma(K))$ satisfies these two properties henceforth. Note that we may also assume that $(Y(K),\gamma(K))$ has no spherical boundary components, by repeated applications of Lemma~\ref{lem:spericalboundarycase}, so we also assume that this holds henceforth. 

Now, by Lemma~\ref{lem:sinhom} we have that $\alpha\cap[\mu_K]=0$ for some $\alpha\in H_2(Y_K,\partial Y_K)$. Let $\alpha$ be such a class. Since $(Y(K),\gamma(K))$ irreducible and $\rank(\SFH(Y(K),\gamma(K)))\neq0$, we have that $(Y(K),\gamma(K))$ is taut by~\cite[Proposition 9.18]{juhasz2006holomorphic}. Lemma~\ref{thm:hierarchy} then states that either $(Y(K),\gamma(K))$ contains a product disk or there is a properly embedded representative $\Sigma$ of $\alpha$ such that $\Sigma$ is the first surface in a sutured manifold hierarchy for $(Y(K),\gamma(K))$ that terminates at the $n$th step with the last decomposition along a surface $\Sigma_n$ that intersects $\partial_KY_K$ in a collection of meridians.
 
We proceed by induction on the length of the sutured hierarchy. We have three cases; either $\Sigma$ is a product disk, $\widehat{\Sigma}_i\cap K=\emptyset$ or $\widehat{\Sigma}_i\cap K\neq\emptyset$, so that $i=n$. Here $\widehat{\Sigma}_i$ indicates the surface in $(Y_{i-1},\gamma_{i-1})$ obtained by capping off $\Sigma_i\subset Y(K)$ with meridional disks in a neighborhood of $K$, which we can do since $\Sigma$ meets the boundary of a neighborhood of $K$ in a collection of meridians.\\

\noindent\textbf{Case 1: $(Y(K),\gamma(K))$ contains a product disk.}  Observe that since $(Y(K),\gamma(K))$ contains a product disk and we are assuming that $(Y(K),\gamma(K))$ is irreducible, $(Y,\gamma)$ also contains a product disk. Now applying~\cite[Lemma 9.13]{juhasz2006holomorphic} to $(Y,\gamma)$ and $(Y(K),\gamma(K))$ implies that $\rank(\SFH(Y_1(K),\gamma_1(K)))=2\rank(\SFH(Y_1,\gamma_1))$, as desired.\\

For the next two cases, we equip $\SFH(Y(K),\gamma(K))$ with the $A_{\alpha}$ grading. Since $\SFH(Y(K),\gamma(K))$ and ${\SFH(Y,\gamma)\otimes V}$ have the same rank --- where $V$ is a rank two vector space supported in $A_{H_K(\alpha)}$-grading $0$ --- it follows from Lemma~\ref{lem:spectral} that $\SFH(Y(K),\gamma(K))\cong \SFH(Y,\gamma)\otimes V$ as an $A_{H_K(\alpha)}$ graded vector spaces. The hypothesis in Lemma~\ref{lem:spectral} that $[\mu_K]$ is non-torsion in $H_1(Y(K))$ can be verified by noting that since $[K]\neq 0$ in $H_1(Y;\Q)/H_1(\partial Y;\Q)$, there is a properly embedded surface $\Sigma$ in $(Y(K),\partial(Y(K)))$ that meets $\partial_K Y(K)$ in a collection of parallel oriented curves with non-zero intersection number with $\mu_K$, viewed as a subset of $\partial_KY(K)$. Consequently $[\mu_K]$ cannot be torsion, since $l[\mu_K]\cap [\Sigma]\neq 0$ for any $l$.\\

\noindent\textbf{Case 2: $\widehat{\Sigma}\cap K=\emptyset$.} In this case, we show that $\rank(\SFH(Y_1(K),\gamma_1(K)))=2\rank(\SFH(Y_1,\gamma_1))$. To do so, we first add appropriate product one handles to $(Y,\gamma)$ and $(Y(K),\gamma(K))$ that, in particular, avoid $\partial \Sigma$, to obtain strongly balanced sutured manifolds $(Y',\gamma')$ and $(Y'(K),\gamma'(K))$. Note that this does not change either of the sutured Floer homology groups. This case now follows from the fact that $\SFH(Y',\gamma')\otimes V$ and $\SFH(Y'(K),\gamma'(K))$ have the same rank in the minimal non-trivial $A_{\alpha)}$ (respectively $A_{H_K(\alpha)}$) grading by Lemma~\ref{lem:spectral}. Noting that we have arranged that $\Sigma$ satisfies the necessary properties in Theorem~\ref{thm:hierarchy}, an application of Juh\'asz' sutured decomposition formula~\cite[Theorem 1.3]{juhasz2008floer} to $(Y',\gamma')$ and $((Y'(K),\gamma'(K))$, shows that $2\rank(\SFH(Y'_1,\gamma'_1))=\rank(\SFH(Y'_1(K),\gamma'_1(K)))$, where $(Y'_1,\gamma'_1)$ is the sutured manifold obtained by decomposing $(Y',\gamma')$ along $\Sigma$. Note that $(Y_1,\gamma_1)$ can be obtained from $(Y'_1,\gamma'_1)$ by decomposing along the product disks that are the co-cores of the product $1$-handles. The desired conclusion now follows from applications of~\cite[Lemma 9.11]{juhasz2006holomorphic} to $(Y'_1,\gamma'_1)$ and $(Y'_1(K),\gamma'_1(K))$.\\

  \noindent\textbf{Case 3: $\widehat{\Sigma}\cap K \neq \emptyset$.} We prove the case in which $(Y,\gamma)$ is strongly balanced, noting that the balanced case can be recovered by adding and removing appropriate product one handles as in the $\widehat{\Sigma}\cap K=\emptyset$ case. Now, as noted in the previous case, the spectral sequence from $\SFH(Y(K),\gamma(K))$ to $\SFH(Y,\gamma)$ collapses immediately, so by Lemma~\ref{lem:spectral} we have that $\SFH(Y(K),\gamma(K))\cong\SFH(Y,\gamma)\otimes V$, where we equip $\SFH(Y,\gamma)$ with the $A_\alpha$-grading and $\SFH(Y(K),\gamma(K))$ with the $A_{H_K(\alpha)}$-grading. We derive a contradiction by showing that the span of the $A_{H_K(\alpha)}$-gradings in which $\SFH(Y(K),\gamma(K))$ has non-trivial support is strictly larger than the span of the with the $A_\alpha$-gradings in which $\SFH(Y,\gamma)$ has non-trivial support. Specifically we show that the minimum $A_{H_K(\alpha)}$ in which $\SFH(Y(K),\gamma(K))$ is non-trivial is strictly smaller than the $A_\alpha$-gradings in which $\SFH(Y,\gamma)$ has non-trivial support and that the maximum $A_\alpha$-gradings in which $\SFH(Y(K),\gamma(K))$ has non-trivial support is at least as large as $\SFH(Y,\gamma)$. By Lemma~\ref{lem:maxagrading}, the minimum non-trivial $A_\alpha$-grading is given by the maximum value of $c(\Sigma,t)$ over all taut decomposing surfaces $\Sigma$. Let $S$ be a surface maximizing this quantity.  Consider the surface $\widehat{S}$ in $(Y,\gamma)$ obtained by capping-off the boundary components of $\partial S$ on $\partial_K$ with meridional disks for $K$. Observe that $c(\widehat{S},\overline{t})=\chi(\widehat{S})+I(\widehat{S})-r(\widehat{S},\overline{t})$. First note that $I(\widehat{S})=I(S)$, since $\widehat{S}$ intersects $\gamma$ as many times as $S$ intersects $\gamma(K)$.  We also have that $r(\widehat{S},\overline{t})=r(S,t)$ since $S$ intersects $\partial_K$ in meridians.  Finally, observe that $\chi(\widehat{S})>\chi(S)$. In particular, it follows that $c(S,t)< c(\widehat{S},\overline{t})$, so that the minimum $A_{H_K(\alpha)}$ in which $\SFH(Y(K),\gamma(K))$ is non-trivial is strictly smaller than the $A_\alpha$-gradings in which $\SFH(Y,\gamma)$ has non-trivial support.
  
  To prove that the maximum $A_\alpha$-gradings in which $\SFH(Y(K),\gamma(K))$ has non-trivial support is at least as large as $\SFH(Y,\gamma)$, one can proceed as above, this time picking a taut surface $S$ that represents $H_K(-\alpha)$ and noting that $c(S,t)\leq c(\widehat{S},\overline{t})$. In more detail, the maximum non-trivial $A_{H_K(\alpha)}$ in which $\SFH(Y(K),\gamma(K))$ is non-trivial is given by \begin{align*}\max\{\langle c_1(\s,t),[H_K(\alpha)]\rangle:\SFH(Y(K),\gamma(K))\neq 0\}&=-\min\{\langle c_1(\s,t),[H_K(-\alpha)]\rangle:\SFH(Y(K),\gamma(K))\neq 0\}\\&= -c(S,t). \end{align*} Here the second inequality follows from Lemma~\ref{lem:maxagrading}. Consider the surface $\widehat{S}$ in $(Y(K),\gamma(K))$ obtained by capping-off the boundary components of $\partial S$ on $\partial_K$ with meridional disks for $K$. Note that unlike in the case of the argument above $S$ may have no boundary components on $\partial_K$. Observe that $c(\widehat{S},\overline{t})=\chi(\widehat{S})+I(\widehat{S})-r(\widehat{S},\overline{t})$, $I(\widehat{S})=I(S)$, and $r(\widehat{S},\overline{t})=r(S,t)$ as before.  Finally, observe that $\chi(\widehat{S})\geq\chi(S)$. It follows that $c(S,t)\leq c(\widehat{S},\overline{t})$, as claimed. In turn we have that the maximum $A_{H_K(\alpha)}$ in which $\SFH(Y(K),\gamma(K))$ is non-trivial is at least the maximum $A_\alpha$-gradings in which $\SFH(Y,\gamma)$ has non-trivial support, as desired.\end{proof}

\section{The link Floer homology of homologically non-trivial links in $S^1\times S^2$}\label{sec:S1S2}

Ni classified null-homologous Floer simple knots in $\#^n(S^1\times S^2)$~\cite[Theorem 1.3]{ni2014homological}. In this section, we give some other classification results for links in $S^1\times S^2$ for which the knot Floer homology takes various special forms. To begin with, recall that $\widehat{\HF}(S^1\times S^2)$ comes equipped with an $A_{S^2}$-grading and
    $\widehat{\HF}(\#^n(S^1\times S^2))$ is rank two and supported in $A_{S^2}$-grading $0$.

Given a spherical braid $\beta$, we can consider the link in $S^1\times S^2$ shown in Figure~\ref{fig:sphericalbraidclosure}, which we are denoting $\widehat{\beta}$. We will be particularly interested in the case of the $1$-braid $\mathbf{1}_1$ --- which we call $K_1$ --- and the $2$-braid $\sigma_1$, where $\sigma_1$ is the standard Artin generator. We call this knot $K_2$. Observe that the knots $\widehat{\sigma_1^{2n+1}}$ are isotopic to $K_2$ for all $n$.

We will be interested in links $L$ that have link Floer homology of minimal rank in the maximal non-trivial $A_{\alpha_i}$-grading. We will have different cases according to the value of the geometric intersection number of $L$ with non-separating spheres in $S^2$. When $L\cap S^2=\emptyset$ we have the following more general Proposition:

\begin{proposition}\label{prop:K1comp}
Suppose that $L$ is an $m$-component link in a $3$-manifold $Y$ and there exists a non-separating sphere $S^2 \subset Y$ such that $L\cap S^2=\emptyset$. Then $Y$ splits as a connect sum $Y\cong (S^1\times S^2)\#Y'$ and ${(L,Y)\cong (\emptyset,S^1\times S^2)\#(L,Y')}$. Moreover $$\widehat{\HFL}(L,Y)\cong \widehat{\HFL}(L,Y')\otimes V$$ where $V$ is a rank two vector space supported in Alexander multi-grading $\mathbf{0}$.
\end{proposition}
   The proof is essentially identical to that of Lemma~\ref{lem:incompressible}.
\begin{proof}
Suppose $L$, $Y$ and $S^2$ are as in the statement of the proposition. Pick an arc $\xi$ that connects the two sides of $S^2$. Observe that $\nu(\xi\cup S^2)$ is a punctured $S^1\times S^2$. It follows that $(L,Y)\cong (S^1\times S^2)\#(L,Y')$ as claimed. Moreover, we have that  $$\widehat{\HFL}(L,Y)\cong\widehat{\HFL}(\emptyset, S^1\times S^2)\otimes\widehat{\HFL}(L,Y')$$ by the K\"unneth formula. But $\widehat{\HFL}(\emptyset, S^1\times S^2)\cong V$ where $V$ is a rank two vector space supported in Alexander multi-grading $\mathbf{0}$, so we have the desired result.
\end{proof}

If $|L\cap S^2|=1$, we have the following similar result:

\begin{proposition}\label{lem:K1component}
Suppose that $L$ is an $m$-component link in a $3$-manifold $Y$ and there exists a non-separating sphere $S^2 \subset Y$ such that $|L\cap S^2|=1$. Then $Y$ splits as a connect sum $Y\cong (S^1\times S^2)\#Y'$ and ${(L,Y)\cong (K_1,S^1\times S^2)\#(L',Y'),}$ where $L'$ is an $m-1$ component link. Moreover $\widehat{\HFL}(L,Y)\cong0$.
\end{proposition}

\begin{proof}
Suppose $L$ and $Y$ are as in the statement of the proposition. Let $K$ be a component of $L$ that meets a non-separating sphere $S^2\subset Y$. Observe that $\nu(K\cup S^2)$ is a punctured $S^1\times S^2$, and that $K$ is a copy of $K_1$ in this $S^1\times S^2$. It follows that $(L,Y)\cong (K_1,S^1\times S^2)\#(L',Y')$ as claimed. Moreover, we have that  $$\widehat{\HFL}(L,Y)\cong\widehat{\HFL}(K_1,S^1\times S^2)\otimes\widehat{\HFL}(L',Y')$$ by the K\"unneth formula. Note too that $\widehat{\HFL}(K_1,S^1\times S^2)\cong 0$, so the result follows.
\end{proof}

We now proceed to investigate non-separating spheres which intersect $L$ in two or more points.  First, we describe a construction of links in $Y\#S^1\times S^2$ from links in $Y$: given a link $L$ in $Y$ with two marked points, and a framed arc $\alpha$ connecting the two points on $L$, one can construct the link $L_{\alpha}$ in $Y\#S^1\times S^2$ by attaching a band along the framed arc $\alpha$ and doing $0$ surgery on a meridian of $\alpha$. See Figure~\ref{fig:bandsurgery}.

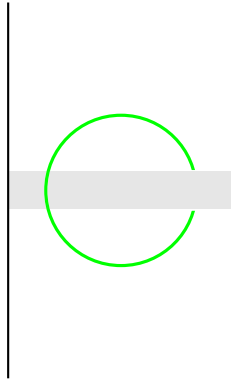
\begin{figure}
   \begin{center}
    \begin{tikzpicture}
    
          \draw[very thick, green] (0,1) arc[start angle=90, end angle=-90, radius=1];

        % Horizontal grey strip

               \draw[very thick, white] (-1.5,-0.25) -- (1.5,-0.25);
               \draw[very thick, white]   (-1.5,0.25) -- (1.5,0.25);
        \fill[gray!20] (-1.5,-0.25) rectangle (1.5,0.25);

           \draw[very thick, white] (0,-1) arc[start angle=270, end angle=90, radius=1];

          \draw[very thick, green] (0,-1) arc[start angle=270, end angle=90, radius=1];

    % Draw vertical arcs

        \draw[thick] (-1.5,-2.5) to[out=90,in=270] (-1.5,2.5);
        \draw[thick] (1.5,-2.5) to[out=90,in=270] (1.5,2.5);

    \end{tikzpicture}
\end{center}
\caption{$0$-surgery on the meridian --- shown in green --- of the band $\alpha$ --- shown in grey produces a knot in a $3$-manifolds with an $S^1\times S^2$ factor.}\label{fig:bandsurgery}
\end{figure}

The content of the following proposition is a generalization of a step in the proof of~\cite[Theorem 2.1]{Wangsplitcharp}.
\begin{proposition}
      Let $S$ be a non-separating sphere in $Y$ and $L$ be an $m$-component link in $Y$ with $|S\cap L|=2$. Then $Y\cong Y'\#S^1\times S^2$ and $L=L'_{\alpha}$ for some link $L'\subset Y'$ and framed arc $\alpha$. Moreover, if $L$ has one more component than $L'$ then
      \begin{align*}
          \widehat{\HFK}(L,Y)\cong\widehat{\HFK}(L',Y')\otimes V^{\otimes 2};\end{align*}

       \noindent  while  if $L$ and $L'$ have the same number of components  \begin{align*}
          \widehat{\HFK}(L,Y)\cong\widehat{\HFK}(L',Y')\otimes V;\end{align*}
          
        \noindent  and if $L'$ has one more component than $L$ then:\begin{align*}
          \widehat{\HFK}(L,Y)\cong\widehat{\HFK}(L',Y').\end{align*}
\end{proposition}

\begin{proof}
    Suppose $S$, $L$ are as in the statement of the theorem. Let $\xi$ be an arc connecting the two sides of $S$. A tubular neighborhood of $S\cup\xi$ is a once punctured copy of $S^1\times S^2$, so that $Y\cong Y'\#(S^1\times S^2)$ for an appropriate $Y'$. It is clear that $L$ is of the form $L'_\alpha$ for some $L'\subset Y'$ and arc $\alpha$ with $\partial\alpha\subset L$. The intersection of $S$ with the exterior of $L$ can be isotoped so that it yields a product annulus in $S^1\times S^2\setminus \nu(L)$. Decomposing along this product annulus we obtain $L'$. The result follows from~\cite[Theorem 1.3]{juhasz2008floer}, and the behavior of sutured Floer homology under the removal of excess parallel sutures.
\end{proof}

We proceed now restrict our attention to links in $S^1\times S^2$ and prove the second main theorem advertised in the introduction, which we recall here for the reader's convenience:

    \sphericalbraiddetection

Observe that if $|S\cap L|=1$ then $L$ has a $K_1$ component and $\rank(\widehat{\HFK}(L))=0$ per Proposition~\ref{lem:K1component}.  If $|S\cap L|=2$ the case in which $L$ has a single component follows from Proposition~\ref{prop:2int}.

For the statement of the following lemma note that if $L'$ is a homologically non-trivial link in $S^1\times S^2$ with irreducible exterior such that $|L'\cap (\{*\}\times S^2)|>1$, then there is taut surface $\Sigma$ with $[\Sigma]=[(\{*\}\times S^2)\setminus\nu(L)]$. This follows from the proof of Theorem~\ref{thm:hierarchy} and the fact that $\partial_*([(\{*\}\times S^2)\setminus\nu(L)])\neq 0\in H_1(\partial(\nu(L)))$. With this fact at hand, we can make sense of the following Lemma:

\begin{lemma}\label{lem:spherebraidlemma}
    Suppose $L'$ is a homologically non-trivial link in $S^1\times S^2$ with irreducible exterior such that $|L'\cap (*\times S^2)|>1$. Let $(Y,\gamma)$ be the sutured manifold obtained by decomposing $((S^1\times S^2)(L'),\gamma(L'))$ along a taut surface $S$ representing the image of $[\{*\}\times S^2]\in H_2((S^1\times S^2)(L'),\partial(\nu(L')))$ with a maximal number of boundary components, removing any excess parallel sutures, then filling in $\nu(L)$, where $L$ is the $n$-component sub-link of $L'$ consisting of components that do not intersect $S$. Then ${\rank(\SFH(Y(L),\gamma(L)))\geq 2^n}$, with equality if and only if $n=0$, in which case $(Y,\gamma)$ is a product sutured manifold.
\end{lemma}

This result is essentially a modified version of~\cite[Proposition 5.5]{binns2024floer}, which can be thought of as classifying links $L$ in sutured manifolds $(Y,\gamma)$ obtained by decomposing a links $L'$ along a fixed longitudinal surface for some component, where $(Y(L),\gamma(L))$ is of minimal rank or lower. Indeed, for the proof, we follow the proof of~\cite[Proposition 5.5]{binns2024floer}, which is more straightforward in the case at hand, since we are treating the minimal rank case, rather than the next-to-minimal rank case.
\begin{proof}
   Note that $(Y(L),\gamma(L))$ embeds in $S^3$. To see this, one can view $S^1\times S^2$ as $0$-surgery along an unknot, $\iota$, and isotope $L'$ off the handle attached to $\iota$. $(Y(L),\gamma(L))$ can be obtained by removing a neighborhood of $\iota$ from $(S^1\times S^2(L'),\mu_{L'})$, decomposing $(S^1\times S^2(L'),\mu_{L'})$ along a longitudinal surface for $\iota$ to obtain a sutured manifold $(Y',\gamma')$ then attaching a collection of thickened disks to a collection of curves isotopic to the suture on $(Y',\gamma')$ corresponding to $\iota$.

     Suppose, towards a contradiction, that $(Y(L),\gamma(L) )$ is irreducible and $${\rank(\SFH(Y(L),\gamma(L))))\leq 2^n}.$$ We proceed by induction on $n$ and the number of essential product annuli in $(Y(L),\gamma(L))$ with both boundary components contained in $(Y,\gamma)$.

    If $n=0$ then the result holds by assumption that the decomposing surface is taut~\cite[Theorem 1.4]{juhasz2008floer}, and the fact that $S^1\times S^2$ admits a unique structure as surface bundle over $S^1$.

Suppose now that $n\geq1$. Observe that since no component of $L$ intersects the decomposing surface representing $\{*\}\times S^2$, each component $L_i$ of $L$ is null-homologous in $S^1\times S^2$, and hence $$[L_i]=0\in H_1(Y(L\setminus L_i))/H_1(\partial (Y(L\setminus L_i))).$$ Thus, we may apply Lemma~\ref{lem:spectral}. If there is some $n-1$ component sublink of $L$, $L''$, such that $\SFH(Y(L''),\gamma(L''))$ is taut, then by inductive hypothesis $\rank(\SFH(Y(L''),\gamma(L'')))\geq 2^{n-1}$. It follows that $L\setminus L''$ is Floer simple in $\SFH(Y(L''),\gamma(L'')))$, so the result follows from the main theorem, Theorem~\ref{thm:main}, together with the fact that $Y(L)$ is irreducible.
    
   Suppose now that no $n-1$ component sublink $L''$ of $L$,  has the property that $\SFH(Y(L''),\gamma(L''))$ is taut. Then the spectral sequence from $\SFH(Y(L),\gamma(L))$ to $\SFH(Y(L''),\gamma(L''))\otimes V$ kills every generator, for every $L''$. It follows from~\cite[Corollary 5.27]{binns2024floer} that $\dim(P(Y(L''),\gamma(L'')))\leq n$, where $P(Y,\gamma)$ denotes the sutured Floer polytope of $(Y,\gamma)$; see~\cite{juhasz2010sutured} for details.
    By repeated applications of~\cite[Lemma 2.17]{binns2024floer} we can decompose $(Y(L),\gamma(L))$ along a finite number of product annuli until the genus of the boundary of the resulting sutured manifold equals the dimension of its sutured Floer homology polytope. We investigate the behavior of such product annuli.

    Let $A$ be a product annulus in $(Y(L),\gamma(L))$. First note that $A$ cannot have boundary components on components of $\partial (Y(L))$ corresponding to distinct components of $L$, as else the meridian of one component of $L$ would be isotopic to that of another component, a contradiction, since $L\subset Y\subset S^3$.

    Suppose $A$ has one boundary component on a component of $\partial(Y(L))$ corresponding to a link component and another boundary component on $\partial Y$. Then there is another taut surface representing $${[\{*\}\times S^2]\in H_2(S^1\times S^2(L'),\partial(\nu(L')))}$$ with a larger number of boundary components, namely $(\Sigma\setminus \nu(\partial A))\cup A_\pm$, where $A_\pm$ are appropriate push-offs of $A$ in the positive and negative normal direction, contradicting our assumption.

If $\partial A$ is contained in a component of $\partial(Y(L))$ corresponding to a single component of $L$, then again using the fact that $S^3$ is atoroidal, we can decompose $(Y(L),\gamma(L))$ along $A$ to obtain the disjoint union of an $m\geq 2$ component exterior of a link $L_1$ in $S^3$ and an $n-m+1$ component link $L_2$ in $Y$. It follows that \begin{align*}
    \rank(\widehat{\HFK}(L_1,S^3))\cdot\rank(\SFH(Y(L_2),\gamma(L_2))\leq 2^n.
\end{align*}

But we also have that  $\rank(\widehat{\HFK}(L_1))\geq 2^m$, so that ${\rank(\SFH(Y(L_2),\gamma(L_2))\leq 2^{n-m}}$. This contradicts the inductive hypothesis, since $L_2$ has $n-m+1$ components.

Suppose now that $A$ is an essential product annulus with both boundary components in $\partial Y$. Then decomposing $(Y(L),\gamma(L)))$ along $A$ yields the sutured exterior of $L$ viewed as a link in some sutured manifold $(Y',\gamma')$. In fact, we have $\rank(\SFH(Y(L),\gamma(L)))\leq 2^n$, which allows us to proceed by induction.

  Thus, it remains only to treat the case in which there is no product annulus in $(Y(L),\gamma(L))$ with a boundary component in $\partial_L$. Note again that $(Y(L),\gamma(L))$ cannot have spherical boundary components since we are assuming $L'$ has irreducible exterior. Observe then that $g(\partial Y(L))> n$. If no $n-1$ component sublink $L''$ of $L$,  has the property that $\SFH(Y(L''),\gamma(L''))$ is taut. Then the spectral sequence from $\SFH(Y(L),\gamma(L))$ to $\SFH(Y(L''),\gamma(L''))\otimes V$ kills every generator, for every $L''$. It follows from~\cite[Corollary 5.27]{binns2024floer} that $\dim(P(Y(L''),\gamma(L'')))\leq n$, a contradiction. Thus at least one $n-1$ component sublink of $L$ has the property that $\SFH(Y(L''),\gamma(L''))\not\cong 0$. By inductive hypothesis $\rank(\SFH(Y(L''),\gamma(L'')))\geq 2^{n-1}$, so that $L$ is Floer minimal in $(Y,\gamma)$. The result now follows from an application of Theorem~\ref{thm:main}. \end{proof}

    \begin{proof}[Proof of Theorem~\ref{thm:sphericalbraiddetection}]
        This is a consequence of Lemma~\ref{lem:spherebraidlemma}.
    \end{proof}

    We can do slightly better in the special case that $|L'\cap(\{*\}\times S^2)|=2$, where we do not need to require that $L$ is non-null-homologous.
\begin{corollary}\label{prop:2int}
    Let $S$ be a non-separating sphere in $S^1\times S^2$ and $K$ be a knot in $S^1\times S^2$ with $|S\cap K|=2$ and $\rank(\widehat{\HFK}(K,S^1\times S^2))=2$. Then $K$ is $K_2$.
\end{corollary}

\begin{proof}
Observe that $(S^1\times S^2(L),\mu_L)$ contains an embedded annulus $A$ representing $[S\setminus \nu(L)]$. Decompose $(S^1\times S^2(L),\mu_L)$ along $A$ and remove excess parallel sutures. Call the resulting sutured manifold $(Y,\gamma)$. Suppose that $(Y,\gamma)$ is not taut. Since $\chi(R_\pm(\gamma))=0$, the surfaces $R_\pm(\gamma)$ are Thurston norm minimizing. Moreover, since $|S\cap K|=2$, and $(S^1\times S^2(L),\mu_L)$ is irreducible and hence so too is $(Y,\gamma)$. Thus $(Y,\gamma)$ is compressible. Since $(Y,\gamma)$ is irreducible, the boundary of any compressing disk can be isotoped to a homologically essential, simple closed curve in $R_\pm(\gamma)\cong A$. This is impossible. Thus $(Y,\gamma)$ is taut and, we see that $\rank(\SFH(Y,\gamma))=1$. Since $(Y,\gamma)$ is irreducible, it is a product sutured manifold. Indeed, we can see that it must be of the form $(A\times[0,1],(\partial A)\times[0,1])$. It follows that $K$ is $\widehat{\beta}$ for some $2$-braid $\beta$. The result follows, since $K_2$ is the unique spherical $2$-braid closure.
\end{proof}

The author and Dey gave a classification of links with link Floer homology of next to minimal rank in the maximal non-trivial grading corresponding to a given component~\cite[Theorem 5.1]{binns2024floer}. It is natural to ask if there is a version of that classification in the context of spherical braids. It is likely that the same strategy as used in~\cite{binns2024floer} could be adapted to the current setting, but since the proof would necessarily be more involved than in the minimal rank case presented above, we do not pursue this here.

We proceed instead to classify homologically non-trivial links with the same link Floer homology as spherical $3$-braids. We begin The following is a basic result concerning the spherical braid group (with $n=3$, which is the dicyclic group of order $12$). We include it here for the sake of completeness.

\begin{lemma}\label{lem:3braidss1s2}
    Let $\beta$ be a $3$-braid. Then $\widehat{\beta}$ is isotopic to one of $\widehat{\mathbf{1}_3}$, $\widehat{\sigma_1^k}$ for $1\leq k\leq 3$, $\widehat{\sigma_1\sigma_2}$, or $\widehat{\sigma_1^3\sigma_2}$ all of which are non-isotopic.
\end{lemma}

\begin{proof}
    We start with an arbitrary spherical $3$-braid, $\beta$, written in terms of the standard Artin generators --- $\sigma_1,\sigma_2$ --- and show that it can be rewritten, up to conjugation, as a word of the desired form.
    
    First note that in the spherical $3$-braid group $\sigma_i^4=\mathbf{1}_3$. Thus $\widehat{\beta}$ is isotopic to $\widehat{\beta'}$ with $\beta'$ a braid word containing only positive powers of the standard Artin generators. We also have that $\sigma_1^2\sigma_2^2=1$, so we can remove any occurrence $\sigma_2^2$ from $\beta'$. Thus $\beta'$ can be written as a product of words of the form $\sigma_i^k$ or $\sigma_1^k\sigma_2$ with $1\leq k\leq 3$. Observe that if $\sigma_1$ appears in $\beta'$ we can take it to be the first letter by an isotopy or conjugation. Thus unless $\beta'$ is $\mathbf{1}_1$, $\sigma_1^k$ for $1\leq k\leq 3$ or $\sigma_2$, we have that $\beta'$ can be written as a product of words $w_k:\sigma_1^k\sigma_2$ with $1\leq k\leq 3$. Observe now that $\sigma_2$ is conjugate to $\sigma_1$, so it remains only to consider product of words $w_k:\sigma_1^k\sigma_2$ with $1\leq k\leq 3$.

Observe that $\sigma_1^2$ commutes with $\sigma_2$, so we may assume that the first word $w_k$ in our expression for $\beta'$ have $k=1$. Since $\sigma_1\sigma_2\sigma_1=\sigma_2\sigma_1\sigma_2$, for $m>0$, $w_k(\sigma_1\sigma_2)^m$ is equivalent to the words to a shorter word. Thus we are left only with the words $w_k:=\sigma_1^k\sigma_2$ with $1\leq k\leq 3$. To conclude, observe that in the spherical $3$-braid group $\sigma_1^2\sigma_3$ is conjugate to $\sigma_2^3$, which is conjugate to $\sigma_1^3$, concluding the proof, since we the dicyclic group of order twelve has six conjugacy classes.\end{proof}

Note in particular that there are two non-isotopic spherical $3$-braid closures with $n$ components for ${1\leq n\leq 3}$. 
Martin used classifications of (classical) $3$-braid  representations of the unknot in her proof that Khovanov homology detects $T(2,6)$~\cite{martin2022khovanov}. Similar ideas were used to deduce detection results in~\cite{binns2020knot,binns2024closures}. Likewise, Lemma~\ref{lem:3braidss1s2} enables us to deduce some link detection results. For the statement --- and the proof --- of these results we note that for an $m$-component link, $L$, $\widehat{\HFL}(L)$ carries another $m-1$-gradings (in addition to the $A_{[\Sigma]}$-grading, where $\Sigma$ is the image of $\{*\}\times S^2$ in the exterior of $L$) corresponding to evaluating relative Chern classes on homology classes $\alpha\in H_2((S^1\times S^2)\setminus \nu(L),\partial (\nu(L)))$ with non-trivial image in $H_1(\partial (\nu(L)))$.

\sphericalthreebraids

Note that the split case follows readily from this result, the K\"unneth formula (i.e. Equation~(\ref{eq:Kunneth})) along with  Ni's classification of link with link Floer homology of minimal rank in $S^3$~\cite[Proposition 1.4]{ni2014homological}.

Observe that, if $n=3$ in the statement of the Corollary, the maximal $A_{[\Sigma]}$-grading in which ${\widehat{\HFK}(L,S^1\times S^2)}$ is non-trivial is one and that ${\rank(\widehat{\HFK}(L,S^1\times S^2;A_{[\Sigma]}=1))=2^{{m}}}$, where $m\leq 3$ is the number of components of $L$.

\begin{proof}
Suppose $L$ is a homologically non-trivial $m-$component link with the link Floer homology of $\widehat{\beta}$. If the exterior of $L$ is reducible, then there is an embedded $2$-sphere $S$ in the exterior of $L$. If $S$ is non-separating, then    $L$ would be homologically trivial, a contradiction. If $S$ is separating then $L$ would be split contradicting our assumption.

We have three cases according to the value of $n$. If $n=1$ then $\widehat{\HFL}(\widehat{\beta}=0)$, so that the sutured exterior of $L$ is not taut and hence compressible. It follows that $L$ is the spherical one braid. If $n=2$ then the the maximal non-trivial $A_{[\Sigma]}$-grading of $\widehat{\beta}$ is $0$, and $\rank(\widehat{\HFK}(L,S^1\times S^2;A_{[\Sigma]}=0))=2^m$. It follows that $|L\cap S^2|=2$ since sutured Floer homology detects the Thurston norm of link exteriors~\cite[Proposition 7.7]{friedl2011decategorification}, and consequently that $L$ is a spherical braid by Theorem~\ref{thm:sphericalbraiddetection}. It thus remains to deal with the case that $n=3$. We duly assume that $n=3$ henceforth.
    
    Now, since $\rank(\widehat{\HFK}(L,S^1\times S^2;A_{[\Sigma]}=1))=2^{{m}}$ and $L$ is homologically non-trivial, Theorem~\ref{thm:sphericalbraiddetection} implies that there is a representative $\Sigma$ of $(*\times S^2)\setminus \nu(L)$ such that decomposing $((S^1\times S^2)(L),\mu(L))$ along $\Sigma$ yields a product sutured manifold. It follows that $L$ is of the form $\widehat{\beta}$ for some spherical braid $\beta$. Since the maximal $A_{[\Sigma]}$-grading in which $\rank(\widehat{\HFK}(L,S^1\times S^2)$ is non-trivial is one, it follows from the fact that sutured Floer homology detects the Thurston norm of link exteriors~\cite[Proposition 7.7]{friedl2011decategorification} that $\Sigma$ has Euler characteristic $-1$. It follows that $\beta$ is a $3$-braid or $\beta$ is a one braid in a torus-bundle over $S^1$. Since $S^1\times S^2$ is not a torus-bundle over $S^1$  --- which follows, for example, from the fact that $S^1\times S^2$ is reducible while torus bundles have universal cover $\R^3$ and are hence irreducible --- the result follows.\end{proof}

\bibliographystyle{plain}
\bibliography{bibliography}

@article{HolomorphicdiskslinkinvariantsandthemultivariableAlexanderpolynomial,
	title = {Holomorphic disks, link invariants and the multi-variable {Alexander} polynomial},
	volume = {8},
	issn = {1472-2739},
	url = {https://msp.org/agt/2008/8-2/p01.xhtml},
	doi = {10.2140/agt.2008.8.615},
	number = {2},
	urldate = {2020-06-30},
	journal = {Algebraic \& Geometric Topology},
	author = {Ozsv\'{a}th, Peter and Szab\'{o}, Zolt\'{a}n},
	month = may,
	year = {2008},
	note = {Publisher: Mathematical Sciences Publishers},
	pages = {615--692}
}

@article{li2022floer,
  title={On {F}loer minimal knots in sutured manifolds},
  author={Li, Zhenkun and Xie, Yi and Zhang, Boyu},
  journal={Transactions of the American Mathematical Society, Series B},
  volume={9},
  number={17},
  pages={499--516},
  year={2022}
}

@article{Holomorphicdisksandknotinvariants,
	title = {Holomorphic disks and knot invariants},
	volume = {186},
	issn = {0001-8708},
	url = {http://www.sciencedirect.com/science/article/pii/S0001870803002330},
	doi = {10.1016/j.aim.2003.05.001},
	language = {en},
	number = {1},
	urldate = {2020-06-30},
	journal = {Advances in Mathematics},
	author = {Ozsv\'{a}th, Peter and Szab\'{o}, Zolt\'{a}n},
	month = aug,
	year = {2004},
	pages = {58--116}
}

@article{ozsvath2004holomorphic,
  title={Holomorphic disks and topological invariants for closed three-manifolds},
  author={Ozsv{\'a}th, Peter and Szab{\'o}, Zolt{\'a}n},
  journal={Annals of Mathematics},
  pages={1027--1158},
  year={2004},
  publisher={JSTOR}
}

@article{Rasmussen,
	title = {Floer homology and knot complements},
	url = {http://arxiv.org/abs/math/0306378},
	urldate = {2020-06-30},
	journal = {arXiv:math/0306378},
	author = {Rasmussen, Jacob},
	month = jun,
	year = {2003},
	note = {arXiv: math/0306378}
}

@article{martin2022khovanov,
  title={Khovanov homology detects $ {T} (2, 6) $},
  author={Martin, Gage},
  journal={Mathematical Research Letters},
  volume={29},
  number={3},
  pages={835--850},
  year={2022},
  publisher={International Press of Boston}
}

@article{kronheimer2011knot,
  title={Knot homology groups from instantons},
  author={Kronheimer, Peter B and Mrowka, Tomasz S},
  journal={Journal of Topology},
  volume={4},
  number={4},
  pages={835--918},
  year={2011},
  publisher={London Mathematical Society}
}

@article{hedden2011floer,
  title={On {F}loer homology and the {B}erge conjecture on knots admitting lens space surgeries},
  author={Hedden, Matthew},
  journal={Transactions of the American Mathematical Society},
  volume={363},
  number={2},
  pages={949--968},
  year={2011}
}

@article{friedl2011decategorification,
  title={The decategorification of sutured {F}loer homology},
  author={Friedl, Stefan and Juh{\'a}sz, Andr{\'a}s and Rasmussen, Jacob},
  journal={Journal of Topology},
  volume={4},
  number={2},
  pages={431--478},
  year={2011},
  publisher={Wiley Online Library}
}

@article{scharlemann1989sutured,
  title={Sutured manifolds and generalized {T}hurston norms},
  author={Scharlemann, Martin},
  journal={Journal of Differential Geometry},
  volume={29},
  number={3},
  pages={557--614},
  year={1989},
  publisher={Lehigh University}
}

@incollection{M,
	title = {Exchangable {Braids}},
	isbn = {978-0-521-26982-7},
	language = {en},
	booktitle = {Low {Dimensional} {Topology}},
	publisher = {Cambridge University Press},
	author = {Morton, H. R.},
	collaborator = {Fenn, Roger and Hitchin, N. J.},
	month = jul,
	year = {1985},
	note = {}
}

@article{kim2020links,
  title={Links of Second Smallest Knot {F}loer Homology},
  author={Kim, Juhyun},
  journal={arXiv preprint arXiv:2011.11810},
  year={2020}
}

@article{binns2020knot,
    AUTHOR = {Binns, Fraser and Martin, Gage},
     TITLE = {Knot {F}loer homology, link {F}loer homology and link
              detection},
   JOURNAL = {Algebr. Geom. Topol.},
  FJOURNAL = {Algebraic \& Geometric Topology},
    VOLUME = {24},
      YEAR = {2024},
    NUMBER = {1},
     PAGES = {159--181},
      ISSN = {1472-2747,1472-2739},
   MRCLASS = {57K10 (57K18)},
  MRNUMBER = {4721366},
       DOI = {10.2140/agt.2024.24.159},
       URL = {https://doi.org/10.2140/agt.2024.24.159},
}

@article{kronheimer2011khovanov,
  title={Khovanov homology is an unknot-detector},
  author={Kronheimer, Peter B and Mrowka, Tomasz S},
  journal={Publications math{\'e}matiques de l'IH{\'E}S},
  volume={113},
  number={1},
  pages={97--208},
  year={2011},
  publisher={Springer}
}

@incollection {Floerinstatonhomologysurgeryknots,
    AUTHOR = {Floer, Andreas},
     TITLE = {Instanton homology, surgery, and knots},
 BOOKTITLE = {Geometry of low-dimensional manifolds, 1 ({D}urham, 1989)},
    SERIES = {London Math. Soc. Lecture Note Ser.},
    VOLUME = {150},
     PAGES = {97--114},
 PUBLISHER = {Cambridge Univ. Press, Cambridge},
      YEAR = {1990},
      ISBN = {0-521-39978-5},
   MRCLASS = {57R57 (57N10 57R65 58G05)},
  MRNUMBER = {1171893},
MRREVIEWER = {Ronald\ J.\ Stern},
}

@article{ni2014homological,
  title={Homological actions on sutured {F}loer homology},
  author={Ni, Yi},
  journal={Mathematical Research Letters},
  volume={21},
  number={5},
  pages={1177--1197},
  year={2014},
  publisher={International Press}
}

@article{kronheimer2010knots,
  title={Knots, sutures, and excision},
  author={Kronheimer, Peter and Mrowka, Tomasz},
  journal={Journal of Differential Geometry},
  volume={84},
  number={2},
  pages={301--364},
  year={2010},
  publisher={Lehigh University}
}

@article{juhasz2008floer,
  title={Floer homology and surface decompositions},
  author={Juh{\'a}sz, Andr{\'a}s},
  journal={Geometry \& Topology},
  volume={12},
  number={1},
  pages={299--350},
  year={2008},
  publisher={Mathematical Sciences Publishers}
}

@article {Wangsplitcharp,
    AUTHOR = {Wang, Joshua},
     TITLE = {Split link detection for {$\mathfrak{sl}(P)$} link homology in characteristic {$P$}},
   JOURNAL = {J. Topol.},
  FJOURNAL = {Journal of Topology},
    VOLUME = {16},
      YEAR = {2023},
    NUMBER = {2},
     PAGES = {806--821},
      ISSN = {1753-8416,1753-8424},
   MRCLASS = {57K18},
  MRNUMBER = {4637977},
MRREVIEWER = {William\ Rushworth},
       DOI = {10.1112/topo.12297},
       URL = {https://doi.org/10.1112/topo.12297},
}

@article{binns2024closures,
  title={Closures of 3-braids and detection},
  author={Binns, Fraser},
  journal={Pacific Journal of Mathematics},
  volume={340},
  number={1},
  pages={1--36},
  year={2025},
  publisher={Mathematical Sciences Publishers}
}

@article{binns2024floer,
  title={Floer homology, clasp-braids and detection results},
  author={Binns, Fraser and Dey, Subhankar},
  journal={arXiv preprint arXiv:2405.11224},
  year={2024}
}

@article{juhasz2006holomorphic,
  title={Holomorphic discs and sutured manifolds},
  author={Juh{\'a}sz, Andr{\'a}s},
  journal={Algebraic \& Geometric Topology},
  volume={6},
  number={3},
  pages={1429--1457},
  year={2006},
  publisher={Mathematical Sciences Publishers}
}

@article {FLoerinstanton3manifolds,
    AUTHOR = {Floer, Andreas},
     TITLE = {An instanton-invariant for {$3$}-manifolds},
   JOURNAL = {Comm. Math. Phys.},
  FJOURNAL = {Communications in Mathematical Physics},
    VOLUME = {118},
      YEAR = {1988},
    NUMBER = {2},
     PAGES = {215--240},
      ISSN = {0010-3616,1432-0916},
   MRCLASS = {57N10 (58G05 58G10 58G25)},
  MRNUMBER = {956166},
MRREVIEWER = {Ronald\ J.\ Stern},
       URL = {http://projecteuclid.org/euclid.cmp/1104161987},
}

@article{ozsvath2004genusbounds,
  title={Holomorphic disks and genus bounds},
  author={Ozsv{\'a}th, Peter and Szab{\'o}, Zolt{\'a}n},
  journal={Geometry \& Topology},
  volume={8},
  number={1},
  pages={311--334},
  year={2004},
  publisher={Mathematical Sciences Publishers}
}

@article{gabai1983foliations,
  title={Foliations and the topology of 3-manifolds},
  author={Gabai, David},
  journal={Journal of Differential Geometry},
  volume={18},
  number={3},
  pages={445--503},
  year={1983},
  publisher={International Press of Boston, Inc.}
}

@article{juhasz2010sutured,
  title={The sutured {F}loer homology polytope},
  author={Juh{\'a}sz, Andr{\'a}s},
  journal={Geometry \& Topology},
  volume={14},
  number={3},
  pages={1303--1354},
  year={2010},
  publisher={Mathematical Sciences Publishers}
}

\end{document}